\documentclass[12pt,reqno]{amsart}

\usepackage[latin1]{inputenc}
\usepackage{amsmath}
\usepackage{amsfonts}
\usepackage{amssymb}
\usepackage{graphics}

\usepackage{enumerate}
\usepackage{amssymb,amsmath,amsthm,amscd}
\usepackage{latexsym,verbatim,graphicx,amsfonts}

\usepackage{amscd}
\usepackage{amsmath}
\usepackage{amssymb}
\usepackage{amsthm}
\usepackage{latexsym}
\usepackage{verbatim}

\theoremstyle{plain}
\newtheorem{theorem}{Theorem}[section]

\newtheorem{corollary}[theorem]{Corollary}
\newtheorem{lemma}[theorem]{Lemma}
\newtheorem{prop}[theorem]{Proposition}
\theoremstyle{definition}

\theoremstyle{remark}

\newcommand{\vertiii}[1]{{\left\vert\kern-0.25ex\left\vert\kern-0.25ex\left\vert #1
    \right\vert\kern-0.25ex\right\vert\kern-0.25ex\right\vert}}

\newcommand{\lii}{\langle\kern-0.25ex\langle}
\newcommand{\rii}{\rangle\kern-0.25ex\rangle}

\newcommand{\nri}{n\rightarrow\infty}

\newcommand{\bbR}{\mathbb{R}}
\newcommand{\bbC}{\mathbb{C}}

\newcommand{\bbD}{\mathbb{D}}
\newcommand{\bbN}{\mathbb{N}}

\newcommand{\mca}{\mathcal{A}}

\newcommand{\mcj}{\mathcal{J}}
\newcommand{\mcn}{\mathcal{N}}

\newcommand{\mcp}{\mathcal{P}}
\newcommand{\mcl}{\mathcal{L}}

\newcommand{\mcr}{\mathcal{R}}

\newcommand{\mcu}{\mathcal{U}}
\newcommand{\bard}{\overline{\bbD}}

\newcommand{\eitheta}{e^{i\theta}}

\DeclareMathOperator{\Pol}{{\mathcal P}}

\DeclareMathOperator*{\supp}{supp} \DeclareMathOperator*{\Real}{Re}
\DeclareMathOperator*{\Imag}{Im}

\title[Approximants, Jacobi matrices, and Jentzsch-type theorems]{Zeros of optimal polynomial approximants: \\ Jacobi matrices and Jentzsch-type theorems}
\author[B\'en\'eteau]{Catherine B\'en\'eteau}
\address{Department of Mathematics, University of South Florida, 4202 E. Fowler Avenue, Tampa, FL 33620, USA.}
\email{cbenetea@usf.edu}
\author[Khavinson]{Dmitry Khavinson}
\address{Department of Mathematics, University of South Florida, 4202 E. Fowler Avenue, Tampa, FL 33620, USA.}
\email{dkhavins@usf.edu}
\author[Liaw]{Constanze Liaw}
\address{CASPER and Department of Mathematics, Baylor University, One Bear Place \#97328, Waco, TX 76798-7328, USA.}
\email{Constanze$\underline{\,\,\,}$Liaw@baylor.edu}
\author[Seco]{Daniel Seco}
\address{Departament de Matem\`atiques i Inform\`atica, Universitat de
Barcelona, Gran Via de les Corts Catalanes 585, 08007 Barcelona,
Spain.} \email{dseco@mat.uab.cat}
\author[Simanek]{Brian Simanek}
\address{Department of Mathematics, Baylor University, One Bear Place \#97328, Waco, TX 76798-7328, USA.}
\email{Brian$\underline{\,\,\,}$Simanek@baylor.edu}

\date{\today}

\begin{document}

\begin{abstract}
We study the structure of the zeros of optimal polynomial approximants to reciprocals of functions in Hilbert spaces of analytic functions in the unit disk.
In many instances, we find the minimum possible modulus of occurring zeros via a nonlinear extremal problem associated with norms of Jacobi matrices. We examine global properties of these zeros and prove Jentzsch-type theorems describing where they accumulate.  As a consequence, we obtain detailed information regarding zeros of reproducing kernels in weighted spaces of analytic functions.

\end{abstract}

\maketitle

\normalsize

%%%%%%%%%%%%%%%%%%%%%%%%%%%%%%%%%%%%%%%%%
%%%%%%%%%%%%%%%%%%%%%%%%%%%%%%%%%%%%%%%%%
\section{Introduction}\label{Intro}

Our main object of study will be Hilbert spaces of holomorphic functions on the unit disk $\bbD$.  Let
$\omega:=\{\omega_n\}_{n\geq0}$ be a sequence of positive real numbers satisfying
\begin{align}\label{etacond}
\omega_0=1,\qquad\qquad\lim_{\nri}\frac{\omega_n}{\omega_{n+1}}=1.
\end{align}
We denote by $H^2_{\omega}$ the space of all functions $f(z)$ with Maclaurin series
\begin{align}\label{fsum}
f(z)=\sum_{n=0}^{\infty}a_nz^n,\qquad\qquad|z|<1
\end{align}
for which
\[
\|f\|^2_{\omega} := \sum_{n=0}^{\infty}|a_n|^2\omega_n<\infty.
\]
The space $H^2_{\omega}$ is a reproducing kernel Hilbert space and
we denote by $\langle \cdot,\cdot \rangle_\omega$ its inner product.

We emphasize that $H^2_{\omega}$ is a function space with a specific choice of norm given by the sequence $\omega$.  In \cite{BCLSS,BKLSS}, the authors were interested in norm estimates only, while here our results differ substantially for different yet equivalent norms, thus depending drastically on the geometry of the space.
Cyclicity (and the concept of optimal approximant) for these spaces was studied in \cite{FMS}, where the role of the limit \eqref{etacond} is to ensure that functions analytic in a disk larger than the unit disk belong to all the spaces, and that all functions in these spaces are analytic in the unit disk. The most interesting examples of such spaces arise when $\omega_n=(n+1)$ (the Dirichlet space $\mathcal{D}$), when $\omega_n=1$ (the Hardy space
$H^2$), and when $\omega_n=(n+1)^{-1}$ (the Bergman space $\mathcal{A}^2$). For a detailed account of function theory in these classical spaces, we refer the reader to \cite{Du,DS,EKMRBook,HKZ}.  Following
the investigation by several of the authors in \cite{BKLSS}, we will pay special attention to two cases: first, when
$\omega_n=(n+1)^{\alpha}$ for some $\alpha\in\bbR$, in which case $H^2_{\omega}$ will be denoted $D_\alpha$. We call these particular spaces Dirichlet-type spaces. The second group of spaces where $\omega_n = {\beta + n + 1 \choose n}^{-1}$ for some $\beta>-1$ will be called Bergman-type spaces, denoted $H^2_{\omega}
=\mathcal{A}^2_{\beta}$. For this choice of sequence $\omega$, the norm of a function $f$ satisfies
\begin{equation}\label{Bergmannorm}
\|f\|^2_{\omega}= (\beta + 1) \int_{\bbD} |f(z)|^2 (1-|z|^2)^{\beta}dA(z),
\end{equation}
where $dA$ is normalized area measure.  This is easily verified by integration in polar coordinates.

Recall that a function $f$ is cyclic if and only if there exists a sequence of polynomials $\{q_n\}_{n\in\bbN}$ such that the functions $\{q_nf\}_{n\in\bbN}$ converge to $1$ in $H^2_\omega$ as $\nri$.  For example, the function $1$ is cyclic in all these spaces.  For cyclic functions $f$, the polynomial $q_n$ approximates $1/f$ in some sense, even when $1/f$ is not in the space. It is natural then to look for the best approximants.  This motivates the following definition, which need not be restricted to cyclic vectors.

For $f \in H^2_{\omega}$, we say that a polynomial $p_n$ of degree at most $n\in \mathbb{N}$ is an {\it optimal approximant} of order $n$ to $1/f$ if $p_n$ minimizes $\|p f-1\|_\omega$ among all polynomials $p$ of degree at most $n$. The existence and uniqueness of an optimal approximant is clear since the space $f\cdot\mathcal{P}_n$, where $\Pol_{n}$ is the space of all polynomials of degree at most $n$, has finite dimension. Given $f \in H^2_{\omega}$, define $\{\varphi_n\}_{n=0}^{\infty}$ to be the orthonormal polynomials in the ``weighted" space related to $f$, that is, the polynomial $\varphi_k$ has degree exactly $k$ and this sequence satisfies
\[
\langle \varphi_k f , \varphi_j f \rangle_\omega = \delta_{k,j}.
\]
In \cite[Proposition 3.1]{BKLSS}, the authors showed that the optimal polynomial approximants in $D_\alpha$ are related to the orthonormal polynomials. The argument there yields more
 generally that in all
$H_\omega^2$, we have
\begin{equation}\label{optortho}
    p_n(z) =  \overline{f(0)} \sum_{k=0}^n \overline{\varphi_k(0)} \varphi_k(z).
\end{equation}
In particular, in the classical Hardy space $H^2$, the optimal approximants are scalar multiples of the well-known ``reversed" orthonormal polynomials, obtained from $\varphi_k$ by conjugating and reversing coefficients (see, for example, \cite{Ger}). Morover, as noticed in \cite{BKLSS}, since $p_nf$ is the orthogonal projection of $1$ onto the space $f \cdot \Pol_{n}$, then for any polynomial $q \in \Pol_{n}$, we have $\langle qf, p_nf \rangle_{\omega} = \langle qf,
1 \rangle_{\omega} = q(0) f(0)$. Therefore, assuming $f(0) \neq 0,$ the polynomial $ p_n(z)/\overline{f(0)}$ equals $ k_n(z,0),$ where $k_n(z,w)$ is the reproducing kernel for $\Pol_{n}$ in the weighted space with weighted inner product given by $\langle g,h \rangle_{f} := \langle gf,hf\rangle_{\omega}$. Thus, the zeros of the optimal approximants are precisely the zeros of these reproducing kernels $k_n(z,0)$, and are related to the behavior of the zeros of the orthogonal polynomials.  In what follows, for simplicity, we will call these reproducing kernels \emph{weighted reproducing kernels}.

In \cite[Theorem 4.2]{BKLSS}, the authors showed that if $f \in D_{\alpha}$ with $f(0) \neq 0$, and if $\alpha \geq 0,$ then the zeros of the optimal approximants are \emph{all} outside the closed unit disk, while if $\alpha < 0$, the zeros are outside the closed disk centered at the origin of radius $2^{\alpha/2}$, but can indeed penetrate the unit disk.  The authors posed the question of whether $2^{\alpha/2}$ is optimal. Some of our main results are strict
improvements on the bound $2^{\alpha/2}$ (see Theorem \ref{newb}) and on the precise criteria regarding which sequences $\{\omega_n\}_{n\geq0}$ admit a function $f\in H^2_{\omega}$
with an optimal polynomial approximant to $1/f$ that vanishes inside the unit disk (see Corollary \ref{in}).

Motivated by these observations, we are led to the following questions:

\begin{itemize}
\item[(i)]  Given $\omega=\{\omega_n\}_{n\geq0}$, does there exist a
function $f\in H^2_{\omega}$ that has an optimal polynomial approximant
to $1/f$ vanishing inside the unit disk $\bbD$?
\item[(ii)]  If the answer to question (i) is ``yes", then what is
\begin{align}\label{iinf}
\inf_{n \in \bbN} \left\{|z|:p_n(z)=0,
\|p_nf-1\|_\omega=\min_{q\in\mcp_n}\|qf-1\|_\omega,\,f\in
H^2_{\omega}\right\}?
\end{align}
\item[(iii)]  Is the infimum in (\ref{iinf}) a minimum?
\item[(iv)]  If the answer to question (iii) is ``yes", then what are the extremal functions $f\in H^2_{\omega}$ for which the minimum in (\ref{iinf}) is attained?
\item[(v)] When the degree $n \rightarrow \infty$, where do the zeros of the optimal approximants accumulate?  Is there an analogue of the celebrated theorem of Jentzsch \cite{Di,DHK,La,Ti} about the accumulation of the zeros of Taylor polynomials?
\end{itemize}

\noindent The goal of this paper is to answer these questions.

We begin Section \ref{s-EP} by formulating an extremal problem whose
solution will lead us to a resolution of questions (i) through (iv).
We show that the answer to the first question depends on the choice
of $\omega$. More precisely, if there exist $n \geq 1$ and  $k \in
\bbN$ such that $\omega_{k+n} < \omega_{k}/4$, there is a positive
answer to (i), while for any nondecreasing sequence $\omega$ the
answer is negative. We will be able to reduce the problem to
polynomial functions $f$ satisfying a set of necessary conditions
that can be translated into a recurrence relation for their
Maclaurin coefficients. This recurrence relation connects our theory
with that of orthogonal polynomials and therefore, in Section
\ref{Ortho}, we present background information on orthogonal
polynomials and Jacobi matrices. In Section \ref{solve}, following
an idea of McDougall in \cite{McD}, we establish a precise
connection between solutions to our extremal problem and norms of
certain Jacobi matrices. In particular, we answer the questions
(i)-(iv) in terms of the norms of these Jacobi matrices. Section
\ref{exam} is devoted to the study of the particular cases of
$D_{\alpha}$ and $\mathcal{A}^2_{\beta}$ spaces. For the former, we
answer questions (i) and (iii) in terms of the parameter $\alpha$,
and regarding question (ii), we improve the minimal modulus bound
obtained in \cite{BKLSS}.  For the latter spaces, we answer all the
questions (i)-(iv) in terms of the parameter $\beta$. In Section
\ref{Jentzsch}, we answer question (v).  There, we explore the
global structure of the zeros of optimal polynomial approximants and
prove results reminiscent of Jentzsch's Theorem for Taylor
polynomials, yielding information about the limit distributions of
the zeros for large degree polynomial approximants. In Section
\ref{particular}, we settle the question of explicitly determining
the optimal polynomial approximants in the Hardy space $H^2$ for the
function $f(z) = (1-z)^{a}$ for $\mbox{Re} \, a > 0$, thus resolving
earlier inquiries made in \cite{BKLSS,Seco}. We conclude in Section
\ref{ConcRem} with some remarks concerning future research.

%%%%%%%%%%%%%%%%%%%%%%%%%%%%%%%%%%%%%%%%%

\bigskip

\noindent\textbf{Acknowledgements.}  The authors would like to thank A. Sola for useful discussions, and the National Science Foundation for their support of the SEAM conference in 2016, where some of this work was carried out.  B\'en\'eteau and Khavinson are grateful to the Centre de Recherches Math\'{e}matiques in Montr\'{e}al for hosting them during the spring semester of 2016. The work of Liaw was partially supported by grant \#426258 from the Simons Foundation. Seco acknowledges support from Ministerio de Econom\'ia y Competitividad Project MTM2014-51824-P and from Generalitat de Catalunya Project 2014SGR289.

\section{An Extremal Problem}\label{s-EP}

In this section, we give a precise formulation of an extremal
problem that will lead us to the solutions of questions (i) through
(iv) for the spaces $H^2_{\omega}$, where
$\omega=\{\omega_n\}_{n\geq0}$ is a fixed sequence of positive real
numbers satisfying (\ref{etacond}). Let $f \in H^2_{\omega}$.  From
the construction of the optimal approximants using the Gram-Schmidt
process as described in \cite{BCLSS}, one can show that the zero
$z_1$ of the first order optimal approximant $p_1$ is given by
\[
z_1 =  \frac{\| z f \|_\omega^2}{ \langle f, zf \rangle_\omega }
\]
as long as the denominator is not zero (see \cite[Lemma 7.1]{BKLSS} for details). Furthermore, by absorbing zeros of an approximating polynomial into the function $f$, it is easy to see that the answer to question (i) in Section \ref{Intro} can be reduced to studying zeros of first order approximants only.  Hence, to answer question (ii) from Section \ref{Intro} it suffices to calculate the following quantity:
\begin{equation}\label{extremal}
    \mcu_\omega:=\sup_{f \in H^2_{\omega}}  \frac{| \langle f, zf \rangle_\omega |}{\| z f \|_\omega^2}.
\end{equation}
The solution to question (ii) from Section \ref{Intro} is then given by (see \cite{Seco})
\[
\inf_{n\in\bbN}\left\{|z|:p_n(z)=0,\|p_nf-1\|_\omega=\min_{q\in\mcp_n}\|qf-1\|_\omega,\,f\in H^2_{\omega}\right\}=\frac{1}{\mcu_\omega}.
\]
Note that condition \eqref{etacond} on $\omega$ ensures that
 $\mcu_\omega < \infty.$

Non-linear extremal problems similar to (\ref{extremal}) in $D_\alpha$ for $\alpha\ge 0$ have been previously considered by Fisher and McDougall in \cite{Fish,McD}.  In \cite{McD}, McDougall noticed interesting connections between the extremal problem she considered, orthogonal polynomials, and Jacobi matrices.  This connection is also relevant to our investigation and will be essential to some of our main results, which we now formulate.

\begin{theorem}\label{dalpha}
Consider the extremal problem of finding $\mcu_{\omega}$ given by \eqref{extremal}.
\begin{itemize}
\item[(A)] If there exist $n, k \in \bbN$ such
that $\omega_{k+n+1} < \omega_{k+1}/4$, then $\mcu_{\omega}>1$ and
the supremum defining $\mcu_\omega$ is actually a maximum, i.e., an
extremal function exists.
\item[(B)] If $\omega$ is nondecreasing, then $\mcu_\omega=1$ and the supremum defining $\mcu_\omega$ is not a maximum, i.e., there is no extremal function.
\end{itemize}
\end{theorem}

The most difficult part of this theorem is the existence of an extremal in case (A), to which we dedicate Section \ref{solve}. In Corollary \ref{ext2}, we shall find the extremal function in that case. We now give a proof of the remaining statements in Theorem \ref{dalpha}, excluding existence of the extremal in case (A).

\begin{proof}
From the definition of $\mcu_{\omega}$ and the Cauchy-Schwarz
inequality, we obtain that
\[
\mcu_{\omega} \leq \sup_{f \in H^2_{\omega}}
\frac{\|f\|_{\omega}}{\|zf\|_{\omega}}.
\]
Now, if $\omega$ is nondecreasing, the quotient $\frac{\|f\|_{\omega}}{\|zf\|_{\omega}}$
is less than or equal to $1$.  It is also not difficult to see from \eqref{etacond} that
 $\mcu_{\omega}$  is always at least 1.  Therefore, when $\omega$ is nondecreasing,
$\mcu_{\omega} = 1$ and the above inequality is an equality. Moreover, if there were an extremal function $f^*$,
equality would have to hold in the Cauchy-Schwarz inequality, which would imply that $zf^*=tf^*$ for some scalar
$t$.  Since this can only happen when $f^* \equiv 0$, we see that there is no extremal function. This proves (B).

In order to show that in case (A), $\mcu_{\omega} >1$, it is enough to consider the function
\[
f_{k,n} = z^k T_n \left(\frac{1+z}{1-z}\right),
\]
where $T_n(g)$ denotes the Taylor polynomial of $g$ of degree $n$.  We can then compute
\[
\frac{| \langle f_{k,n}, zf_{k,n} \rangle_{\omega} |}{\| z f_{k,n} \|_{\omega}^2}= 1 + \frac{\omega_{k+1}-4 \omega_{n+k+1}}{\omega_{k+1}+ 4\sum_{t=1}^{n} \omega_{t+k+1}}.
\]
The denominator on the right-hand side is a positive number. By the assumption in (A), the numerator is also positive, and hence $\mcu_{\omega}>1$.
\end{proof}

The remaining question of existence in Theorem \ref{dalpha} will require us to recast our extremal problem (\ref{extremal}) in terms of the Maclaurin coefficients $\{a_n\}_{n=0}^{\infty}$ of the function $f$.  Then, given $f$, we have
\begin{align}\label{fq}
\frac{\langle f,zf\rangle_\omega}{\|zf\|_\omega^2}=\frac{\sum_{n=0}^{\infty}\overline{a}_na_{n+1}\omega_{n+1}}{\sum_{n=0}^{\infty}|a_n|^2\omega_{n+1}}.
\end{align}
From this formula, it is easy to see that in solving the extremal problem (\ref{extremal}), it suffices to consider only functions $f$ with non-negative Maclaurin coefficients. Now we will show that in solving the extremal problem (\ref{extremal}), it suffices to only consider polynomials.  More precisely, let us define
\[
\mcu_{\omega,N}:=\sup\left\{\frac{\sum_{n=0}^{N-1}a_na_{n+1}\omega_{n+1}}{\sum_{n=0}^{N}a_n^2\omega_{n+1}}:a_j\in\bbR\mbox{ for } j=0\hdots, N\right\},
\]
where we interpret the quotient as $0$ if each $a_j=0$ for $j=0,\ldots,N$.  Notice that if we replace each $a_j$ in (\ref{fq}) by $|a_j|$, then the quotient does not decrease in absolute value.  Therefore, we could also write
\[
\mcu_{\omega,N}=\sup\left\{\frac{\left|\sum_{n=0}^{N-1}\bar{a}_na_{n+1}\omega_{n+1}\right|}{\sum_{n=0}^{N}|a_n|^2\omega_{n+1}}:(a_j)_{j=0}^N\in\bbC^{N+1}\right\},
\]
which more closely resembles \eqref{extremal}.

We have the following result:

\begin{theorem}\label{finlim}
The supremum defining $\mcu_{\omega,N}$ is a maximum and it holds that
\[
\mcu_\omega=\lim_{N\rightarrow\infty}\mcu_{\omega,N}.
\]
\end{theorem}

\begin{proof}
First let us prove the limit relation.  It is clear that
\[
\lim_{N\rightarrow\infty}\mcu_{\omega,N}=\sup_N\,\mcu_{\omega,N}
\]
and $\mcu_\omega \geq\sup_N\mcu_{\omega,N}$, so we immediately conclude that
\[
\mcu_\omega \geq\lim_{N\rightarrow\infty}\mcu_{\omega,N}.
\]

For the reverse inequality, let $\varepsilon>0$ be fixed.  Let $f\in H^2_{\omega}$ with non-negative Maclaurin coefficients satisfy
\[
\frac{| \langle f, zf \rangle_\omega |}{\| z f
\|_\omega^2}\geq\mcu_\omega-\varepsilon.
\]
The density of polynomials in the space $H^2_{\omega}$ tells us that
\[
\lim_{N\rightarrow\infty}\frac{\langle T_N(f), zT_N(f)
\rangle_\omega}{\| z T_N(f) \|_\omega^2}=\frac{\langle f, zf
\rangle_\omega}{\| z f \|_\omega^2}.
\]
Therefore, by choosing $N$ large enough, we have
\[
\mcu_{\omega,N}\geq\frac{\langle T_N(f), zT_N(f)
\rangle_\omega}{\| z T_N(f)
\|_\omega^2}\geq\mcu_\omega-2\varepsilon.
\]
Since $\varepsilon>0$ was arbitrary, the desired conclusion follows.

To show that the supremum is a maximum, notice that the fraction
\begin{align}\label{thetadef}
\Theta(a_0,\ldots,a_N):=\frac{\sum_{n=0}^{N-1}a_na_{n+1}\omega_{n+1}}{\sum_{n=0}^{N}a_n^2\omega_{n+1}}
\end{align}
remains invariant if we multiply each $a_j$ by a real number $t\neq0$.  Therefore, in maximizing this quantity, one need only consider $(N+1)$-tuples $(a_j)_{j=0}^{N}$ for which $\max\{|a_j|\}=1$.  Notice that
\[
\left\{(a_0,\ldots,a_N)\in\bbR^{N+1}:\max\{|a_0|,\ldots,|a_N|\}=1\right\}
\]
is a compact set on which the function $\Theta(a_0,\ldots,a_N)$ is continuous, so a maximum is attained.
\end{proof}

With Theorem \ref{finlim} in hand, we can approach the problem of calculating $\mcu_\omega$ by calculating $\mcu_{\omega,N}$ and sending $N\rightarrow\infty$.  To do so, we need the following lemma.

\begin{lemma}\label{nz}
 The function $\Theta$ (defined in (\ref{thetadef})) attains a maximum on the set
\begin{align}\label{maxset}
\left\{(a_0,\ldots,a_N)\in\bbR^{N+1}:\min\{a_0,\ldots,a_N\}>0\right\}.
\end{align}
\end{lemma}

\begin{proof}
We have already observed that $\Theta$ attains its maximum on the set
\[
\left\{(a_0,\ldots,a_N)\in\bbR^{N+1}:a_j\geq 0,\, j=0,1,\ldots,N\right\}.
\]
To show that we can remove the boundary of this set, suppose for contradiction that $\Theta$ attains its maximum at $(a_j^*)_{j=0}^N$, where each $a_j^*\geq0$.  Suppose that $a_k^*=0$, and either $a_{k-1}^*$ or $a_{k+1}^*$ is strictly positive.  A simple computation shows that the gradient in the direction $\delta_{k}$ (i.e.,~the vector with 1 in the $k$-th position and 0 in the others) is strictly positive, which gives us a contradiction.
\end{proof}

Since (\ref{maxset}) is an open set and we are considering an
extremal problem with only finitely many real variables, we can
easily solve it by use of Lagrange multipliers. Proceeding with this
calculation, we let $Q_N$ be a polynomial that is extremal for the
problem defining $\mcu_{\omega,N}$.  Write
\begin{align}\label{qdef}
Q_N(z)=a_0+a_1z+\cdots+a_{N}z^{N}.
\end{align}
The Lagrange multipliers method applied to the numerator in the extremal problem for $\mcu_{\omega,N}$, treating the denominator as a constraint, yields a real number $\lambda$ such that
\begin{align}
\label{ncond}a_j&=\lambda a_{j-1}-\frac{\omega_{j-1}}{\omega_j}a_{j-2},\qquad\qquad j=1,\ldots,N,\\
\label{lastcond}0&=\lambda a_{N}-\frac{\omega_N}{\omega_{N+1}}a_{N-1},
\end{align}
where we adopt the convention that $a_{-1}=0$. In other words, the coefficients of the polynomial $Q_N$ satisfy a three-term recurrence relation.  It is well-known that such sequences have an intimate connection with the theory of orthogonal polynomials on the real line.  The connection between extremal problems similar to (\ref{extremal}) and the theory of orthogonal polynomials was first observed by McDougall in \cite{McD}, and we will further develop
those ideas in what follows.  Before doing so, we  review some of the relevant facts from the theory of orthogonal polynomials.

\section{Orthogonal Polynomials and Jacobi Matrices}\label{Ortho}

Presented below is a broad overview of some of the basic facts in the theory of orthogonal polynomials.  For further information on this topic, we refer the reader to \cite{Chi,Ibook,Nevai,Rice,Szego} and references therein.

\subsection{Favard's Theorem.}
If $\mu$ is a compactly supported probability measure on the real line with infinitely many points in its support, then we can form the associated sequence of monic orthogonal polynomials $\{P_n(x)\}_{n\geq0}$, where $P_n$ has degree exactly $n$, that is, polynomials that satisfy
\begin{align*}
\int_\mathbb{R} P_n(x) P_m(x) \, d \mu(x) = K_n\delta_{nm}
\end{align*}
for some constant $K_n>0$. These polynomials satisfy a three term recurrence relation, which therefore gives rise to two bounded real sequences $\{c_j\}_{j=1}^{\infty}$ and $\{v_j\}_{j=1}^{\infty}$ that satisfy
\begin{align}\label{recur}
xP_n(x)=P_{n+1}(x)+v_{n+1}P_n(x)+c_n^2P_{n-1}(x)
\end{align}
(see \cite[Theorem 1.2.3]{Rice}).  Favard's Theorem (see, for example, \cite[p.~21]{Chi}) tells us that the converse is also true, namely that given two sequences of bounded real numbers as above, with each $c_j>0$,  there exists a
compactly supported measure $\mu$ on the real line whose corresponding monic orthogonal polynomials satisfy this recursion.  This measure can be realized as the spectral measure of a Jacobi matrix $\mcj$, which is a tri-diagonal self-adjoint matrix formed by placing the sequence $\{v_j\}_{j=1}^{\infty}$ along the main diagonal and the sequence $\{c_j\}_{j=1}^{\infty}$ along the first main sub-diagonal and super-diagonal.

From the monic orthogonal polynomials, one can also form the sequence of orthonormal polynomials, denoted by $\{\varphi_n(x)\}_{n\geq0}$, which are given by
\begin{equation}\label{onpdef}
\varphi_0(x)=1,\qquad\qquad\varphi_n(x)=\frac{P_n(x)}{\prod_{j=1}^nc_j},\qquad n\in\bbN.
\end{equation}
These polynomials satisfy the recursion relation
\begin{align}\label{normalrecur}
x\varphi_n(x)=c_{n+1}\varphi_{n+1}(x)+v_{n+1}\varphi_n(x)+c_n\varphi_{n-1}(x),\qquad
n\geq0,
\end{align}
where we again adopt the notation that $\varphi_{-1}(x)=0$.

Favard's Theorem thus places in one-to-one correspondence four classes of objects:
\begin{itemize}
\item  Probability measures on the real line whose support is compact and infinite;
\item  Sequences of monic orthogonal polynomials satisfying a recursion of the form (\ref{recur}) with bounded recursion coefficients;
\item  Tri-diagonal bounded self-adjoint matrices whose
main diagonal is real and whose off-diagonal is strictly positive;
\item  Pairs of bounded sequences of real numbers
$\{c_j\}_{j=1}^{\infty}$ and $\{v_j\}_{j=1}^{\infty}$, where each
$c_j>0$.
\end{itemize}
Therefore, whenever we talk about an object in one of these classes, we can talk about the corresponding element from a different class.

\subsection{Zeros}\label{zero}
One of the keys to our analysis will be an understanding of the zeros of the orthogonal polynomials $\{P_n(x)\}_{n\geq0}$.  In particular, we recall \cite[Theorem 1.2.6]{Rice}, which tells us that the zeros of the polynomial $P_N(x)$ are precisely the eigenvalues of the upper-left $N\times N$ block of the corresponding Jacobi matrix $\mcj$.  One additional fact we will need is that the zeros of the polynomials $P_n$ and $P_{n+1}$ are real and strictly
interlace in the sense that between any two zeros of $P_{n+1}$ there is a zero of $P_n$, and $P_n$ and $P_{n+1}$ do not share any common zeros (see \cite[Section 1.2]{OPUC1}).

For general orthogonal polynomials corresponding to a measure supported on any compact set in the complex plane, it is known that their zeros lie inside the convex hull of the support of the measure of orthogonality.  This result is known as Fej\'{e}r's Theorem and is a simple consequence of the \textit{extremal property}:
\[
\|P_n\|_{L^2(\mu)}=\inf\|z^n+\mbox{ lower order terms}\,\|_{L^2(\mu)}
\]
(for an elegant proof of Fej\'er's Theorem, see the proof of \cite[Theorem 11.5]{Widom}).  We note that the extremal property is valid in any inner product space where one can form the sequence of monic orthogonal polynomials.

\subsection{Poincar\'e's Theorem}
One of the tools we will use in our analysis is Poincar\'e's Theorem, which we now recall (following the presentation in \cite[Theorem 9.6.2]{OPUC2}). Consider an $n^{th}$ order recursion relation
\[
y_{k+n} = \beta_{k,1}y_{k+n-1}+\hdots + \beta_{k,n} y_{k},
\hspace{.3in} k = 0, 1, 2, \ldots,
\]
with initial conditions $y_0, \hdots, y_{n-1}$. Assume that the leading term $\beta_{k,n}\neq 0$. Further assume the existence of $\lim_{k\to\infty}\beta_{k,l}$ for all $l = 1,\hdots, n$ and call the limit $\beta_l$. Then the \emph{Poincar\'e polynomial} is defined by
\[
p(z)=z^n-\beta_1 z^{n-1} - \hdots - \beta_{n-1}z - \beta_n.
\]

\begin{theorem}[Poincar\'e's Theorem]
If the roots of the Poincar\'e polynomial $p$ are all of distinct magnitude, then for all non-trivial initial conditions $(y_0,
\hdots, y_{n-1}) \neq (0,\hdots,0)$ the limit
$$
\lim_{k\to\infty}\frac{y_{k+1}}{y_k}
$$
exists and is equal to one of the roots of $p$.
\end{theorem}

\subsection{Regularity}\label{reg}
As already noted, the asymptotic behavior of the zeros of optimal polynomial approximants is closely related to the asymptotic behavior of orthogonal polynomials.  However, to deduce any meaningful results about orthogonal polynomials, one needs some assumptions about the underlying inner product space.  In \cite{StaTo}, Stahl and Totik developed the notion of \textit{regularity} of a measure, which is among the weakest assumptions one can make to deduce any weak asymptotics of the orthogonal polynomials. Regular measures are difficult to
characterize in general.%, but it is known that if $\mu$ is absolutely continuous with respect to area measure on the unit disk and has non-vanishing weight in an inner neighborhood of the unit circle, then $\mu$ is regular.

The notion of regularity extends to arbitrary compactly supported measures in the complex plane, but for our purposes, it suffices to consider only measures on the closed disk of radius $r>0$ and center $0$ (see Chapters 3 and 4, in particular Corollary 4.1.7 and Section 4.2 of \cite{StaTo}). In this case, the regularity of the measure (on the closed disk of radius $r$ and center $0$) can be characterized in terms of the corresponding sequence of orthonormal polynomials $\{\varphi_n\}_{n\geq0}$ by the following equivalent properties:

\begin{itemize}
\item  The leading coefficient $\kappa_n$ of $\varphi_n$ satisfies
\begin{align*}%\label{kcond}
\lim_{\nri}\kappa_n^{1/n}=\frac{1}{r}.
\end{align*}
\item We have
\[
\lim_{\nri}|\varphi_n(z)|^{1/n}=\frac{|z|}{r},\qquad\qquad|z|>r,
\]
and the convergence is uniform on compact subsets.
\end{itemize}

\noindent\textit{Remark.} Regularity also implies that the normalized zero-counting measure for the polynomial $\varphi_n$ (call it $\nu_n$) satisfies
\[
\lim_{\nri}\int_{\overline{\bbD}}\log|z-w|d\nu_n(w)=\log|z|,\qquad\qquad|z|>r,
\]
and the convergence is uniform on compact subsets.  This condition can also be restated as saying that if $\hat{\nu}_n$ is the balayage of $\nu_n$ onto the circle or radius $r$, then the unique weak* limit of the measures $\{\hat{\nu}_n\}_{n\in\bbN}$ is the logarithmic equilibrium measure of the closed disk of radius $r$ centered at $0$.

\bigskip

In the remaining sections, we will use the above tools to solve the extremal problem \eqref{extremal} and deduce some additional information about zeros of optimal polynomial approximants.

\section{The Extremal Problem and Jacobi matrices}\label{solve}

Now we return to the calculation of the extremal quantity $\mcu_\omega$.  Let $\omega=\{\omega_n\}_{n\geq0}$ be a sequence of positive real numbers satisfying (\ref{etacond}).  Let $\mcj_\omega$ be the Jacobi matrix given by
\begin{align}\label{jsdef}
(\mcj_\omega)_{i,j}=\sqrt{\frac{\omega_j}{\omega_{j+1}}}\quad\mbox{
if }\quad |i-j|=1
\end{align}
and setting all other entries equal to zero. In attempting to
calculate $\mcu_{\omega,N}$, we have already observed in Lemma
\ref{nz} that we may restrict our attention to polynomials with
strictly positive coefficients, and due to the scale invariance of
the functional $\Theta$, we may in fact assume that the constant
term is equal to $1$.  This leads us to the following result.

\begin{prop}\label{extremen}
Let $Q_N(z)=a_0+a_1z+\cdots+a_{N}z^{N}$ be an extremal polynomial for the problem defining $\mcu_{\omega,N}$, having all positive coefficients and satisfying $Q_N(0)=1$. Let $\mcj_\omega$ be the infinite Jacobi matrix defined as in (\ref{jsdef}) and let $\{P_n\}_{n\geq0}$ be the sequence of monic polynomials corresponding to $\mcj_\omega$ via Favard's Theorem.  The coefficients $\{a_j\}_{j=0}^N$ of $Q_N$ satisfy
\[
a_j=P_j(\lambda),\qquad\qquad j=0,1,\ldots,N,
\]
where $\lambda=\max\{x:P_{N+1}(x)=0\}$.
\end{prop}

\noindent\textit{Remark.}  Notice the resemblance of Proposition \ref{extremen} to \cite[Proposition 4.2]{McD}.

\begin{proof}
The polynomials $\{P_n\}_{n=0}^{\infty}$ satisfy the recursion relation \eqref{recur} with $v_n = 0$ and $c_n^2= \frac{\omega_n}{\omega_{n+1}}$.  We have already observed that the coefficients of $Q_N$ must satisfy the relation (\ref{ncond}), where $a_{-1}=0$ and $a_0=1$ (by assumption).  Therefore, $a_j=P_j(\lambda)$ for some real number $\lambda$.  It follows that the condition (\ref{lastcond}) is equivalent to the condition that $0=\lambda P_N(\lambda)-\frac{\omega_N}{\omega_{N+1}}P_{N-1}(\lambda)$.  From the recursion relation \eqref{recur}, this is the same as saying $P_{N+1}(\lambda)=0$.  We have thus determined that for $j=0,\ldots,N$, each $a_j=P_j(\lambda)$ for some real number $\lambda$ for which $P_{N+1}(\lambda)=0$ and $P_j(\lambda)>0$ for $j=0,\ldots,N$.

Since the coefficients $\{a_j\}_{j=0}^N$ satisfy the relations (\ref{ncond}) and (\ref{lastcond}), then it is easy to verify that
\[
\frac{\sum_{n=0}^{N-1}a_na_{n+1}\omega_{n+1}}{\sum_{n=0}^{N}a_n^2\omega_{n+1}}=\frac{\lambda}{2}.
\]
Maximizing this quantity subject to the above constraints shows
\[
\frac{\sum_{n=0}^{N-1}a_na_{n+1}\omega_{n+1}}{\sum_{n=0}^{N}a_n^2\omega_{n+1}}=\frac{1}{2}\max\{x:P_{N+1}(x)=0,\, P_j(x)>0\,\mbox{ for } j=0,\ldots,N\}.
\]
Due to the interlacing property of the zeros of orthogonal polynomials and the fact that $P_j$ is monic for every $j$, we know that $P_j(x)>0$ for all $j\in\{0,\ldots,N\}$ if $x$ is the largest zero of $P_{N+1}$.  Therefore, $\lambda$ must be the largest zero of $P_{N+1}$.
\end{proof}

Since the spectral radius of a bounded self-adjoint matrix is equal to its norm, we deduce (see Section \ref{zero}) that $\mcu_{\omega,N}$ is equal to half the norm of the upper-left $(N+1)\times(N+1)$ block of $\mcj_\omega$.  Taking $N\rightarrow\infty$ and applying Theorem \ref{finlim} and the assumption (\ref{etacond}), we deduce the following theorem.

\begin{theorem}\label{normequ}
Let $\omega:=\{\omega_n\}_{n\geq0}$ be given and satisfy (\ref{etacond}). Let $\mcj_\omega$ be the infinite Jacobi matrix defined as in (\ref{jsdef}) and let $\mcu_\omega$ be given by (\ref{extremal}). Then
\[
\mcu_\omega=\frac{\|\mcj_\omega\|}{2},
\]
where $\|\mcj_\omega\|$ denotes the operator norm of $\mcj_\omega$.
\end{theorem}

This leads us to the following answer to question (i) from Section
\ref{Intro}.

\begin{corollary}\label{in}
(i) Under the assumptions of Theorem \ref{normequ}, there exists a function $f\in H^2_\omega$ with an optimal polynomial approximant to $1/f$ vanishing inside the open unit disk if and only if $\|\mcj_\omega\| > 2$.  Equivalently, there exist weighted reproducing kernels $k_n(z,0)$ with zeros inside the open unit disk if and only if $\|\mcj_\omega\| > 2.$

(ii) If $\|\mcj_\omega\| = 2,$ then all the zeros of optimal polynomial approximants, or equivalently of weighted reproducing kernels $k_n(z,0)$, lie outside the open unit disk.  Moreover, if the shift operator does not decrease the norm, i.e., if $\|zf\|_{\omega} \geq \|f\|_{\omega}$, then the zeros lie outside the closed unit disk. In particular, this occurs in all $D_{\alpha}$ ($\alpha \geq 0$) spaces or, more generally, in any $H^2_{\omega}$ space where $\inf_n \left( \frac{\omega_{n+1}}{\omega_n} \right) \geq 1$.
\end{corollary}

\begin{proof}
First note that, since each irreducible factor $(z-z_0)$ of an optimal approximant $p$ to $1/f$ is an optimal approximant for $(z-z_0)/(pf)$, it suffices to consider zeros of first order optimal approximants.

(i) If $\|\mcj_\omega\| > 2,$ then for all sufficiently large $N$, $\mcu_{\omega,N}=\frac{1}{|z_{0,N}|} > 1$, where $z_{0,N}$ is the zero of the first order optimal approximant of $1/Q_N$ (and hence is a zero of the corresponding weighted reproducing kernel), and therefore $1/Q_N$ has an optimal approximant vanishing inside the disk.  The converse is clear.

(ii) Note that   $\mcu_\omega \leq 1$ if and only if $|z_0| \geq 1,$ where $z_0$ is a zero of any optimal polynomial approximant (or weighted reproducing kernel).  Therefore the first statement in (ii) is immediate.  The fact that the zeros lie outside the closed unit disk in the cases specified follows from the proof of \cite[Theorem 4.2]{BKLSS}, where the sufficient condition in that argument is that $\|zf\|_{\omega} \geq \|f\|_{\omega}$.  Note that if $f(z) =
\sum_{n=0}^{\infty} a_nz^n,$ then
$$
\|zf\|_{\omega}^2 = \sum_{n=0}^{\infty} \omega_{n+1} |a_n|^2 =  \sum_{n=0}^{\infty} \frac{\omega_{n+1}}{\omega_n} \omega_n |a_n|^2 \geq \inf_n \left( \frac{\omega_{n+1}}{\omega_n}\right) \|f\|_{\omega}^2.
$$
The proof is now complete.
\end{proof}

\noindent\textit{Remark.} In Section \ref{exam}, we will give several examples where $\mcu_\omega = 1,$ e.g., all the $D_{\alpha}$ spaces where $\alpha \geq 0$, and examples where $\mcu_\omega > 1,$ e.g., $D_{\alpha}$
spaces where $\alpha < 0$, thus giving particular instances where Corollary \ref{in} applies.

\bigskip

\noindent\textbf{Example.}  Consider the Hardy space $H^2$, where $\omega_n=1$ for all $n\geq0$.  In this case, the Jacobi matrix $\mcj_\omega$ is given by
\[
\mcj_\omega=\mcj_1:=\begin{pmatrix}
0 & 1 & 0 & 0 & \cdots\\
1 & 0 & 1 & 0 & \cdots\\
0 & 1 & 0 & 1 & \cdots\\
0 & 0 & 1 & 0 & \cdots\\
\vdots & \vdots & \vdots & \vdots & \ddots
\end{pmatrix}
\]
and is often called the ``free" Jacobi matrix.  It is well-known that the operator norm of this matrix is $2$.  Indeed, if $\mcl$ denotes the left shift operator and $\mcr$ denotes the right shift operator, then $\mcj_1=\mcl+\mcr$, so clearly $\|\mcj_1\|\leq2$.  Furthermore, by considering trial vectors of the form
\[
\left(\underbrace{\frac{1}{\sqrt{n}},\frac{1}{\sqrt{n}},\frac{1}{\sqrt{n}},\frac{1}{\sqrt{n}},\ldots,\frac{1}{\sqrt{n}}}_{n\textrm{ times}},0,0,0,0,\ldots\right)^\top
\]
one sees at once that $\|\mcj_1\|=2$.  Therefore, in this case $\mcu_\omega=1$ and hence all zeros of optimal polynomial approximants lie outside the open unit disk, which was already shown in \cite{BKLSS} by different methods.  In fact, in \cite{BKLSS}, it was shown that the zeros of optimal approximants lie outside the closed unit disk, which confirms that there is no extremal function in this case, as discussed in Theorem \ref{dalpha}, part (B).

\bigskip

The assumption (\ref{etacond}) on the sequence $\{\omega_n\}_{n\geq0}$ implies that $\mcj_\omega$ is a compact
perturbation of the free Jacobi matrix, and hence the spectrum of $\mcj_\omega$ (denoted $\sigma(\mcj_\omega)$) consists of the interval $[-2,2]$ and possibly a countable number of isolated points in $\bbR\setminus[-2,2]$, which are eigenvalues of $\mcj_\omega$ and whose only possible accumulation points are $\pm2$.  Since the diagonal of $\mcj_\omega$ consists entirely of zeros, it follows from the recursion relation that $P_n(x)$ is an odd function of $x$ when $n$ is odd and an even function of $x$ when $n$ is even.  This implies the zeros of $P_n$ are always symmetric about $0$ and hence $x\in\sigma(\mcj_\omega)$ if and only if $-x\in\sigma(\mcj_\omega)$.  We will use this information to answer questions (iii) and (iv) from Section \ref{Intro}.

Before we do, let us define some notation.  Let $H^2_\omega(\bbR)$ be those functions in $H^2_\omega$ with real Maclaurin coefficients.  For any $f\in H^2_{\omega}(\bbR)\setminus\{0\}$ with Maclaurin coefficients $\{a_n\}_{n=0}^{\infty}$, we define
\[
\Theta\left(\{a_n\}_{n=0}^{\infty}\right):=\frac{\sum_{n=0}^{\infty}a_na_{n+1}\omega_{n+1}}{\sum_{n=0}^{\infty}a_n^2\omega_{n+1}}=\frac{\langle f,zf\rangle_{\omega}}{\|zf\|^2_{\omega}}.
\]
It is easy to see that the Maclaurin coefficients $\{a_n\}_{n=0}^{\infty}$ form a critical point of $\Theta$ precisely when the condition (\ref{ncond}) is satisfied for all $j\in\bbN$ and for some real number $\lambda$.  We have already learned that we can find $\mcu_{\omega}$ by calculating the maximum of $\Theta$, but the following theorem shows that we can in fact deduce much more.

\begin{theorem}\label{geq2}
Let $\omega=\{\omega_n\}_{n\geq0}$ be a sequence of positive real numbers satisfying (\ref{etacond}). Suppose $\mcj_\omega$ defined by (\ref{jsdef}) satisfies $\|\mcj_\omega\|>2$.  The real number $t>2$ is in $\sigma(\mcj_\omega)$ if and only if the function
\[
f(z;t):=\sum_{n=0}^{\infty}P_n(t)z^n
\]
is in $H^2_\omega(\bbR)$, where $\{P_n\}_{n=0}^{\infty}$ is the sequence of monic orthogonal polynomials corresponding to $\mcj_\omega$ via Favard's Theorem.  Furthermore, if $f(z;t)\in H^2_\omega(\bbR)$, then the sequence $\{P_n(t)\}_{n=0}^{\infty}$ is a critical point of $\Theta$ and
\[
\Theta\left(\{P_n(t)\}_{n=0}^{\infty}\right)=\frac{t}{2}.
\]
\end{theorem}

\noindent\textit{Remark.}  Notice that Theorem \ref{geq2} has a similar conclusion to that of \cite[Proposition 3.4]{McD}.

\begin{proof}
Suppose $t\in\sigma(\mcj_\omega)\cap(2,\infty)$ so that $t$ is an eigenvalue of $\mcj_\omega$ and let $v:=(v_0,v_1,v_2,\ldots)^\top$ be a corresponding eigenvector whose first non-zero entry is equal to $1$.  Since $ \mcj_\omega v = tv,$ the structure of the matrix $\mcj_\omega$ tells us that the entries $\{v_j\}_{j\geq0}$ must satisfy
\[
tv_j=\sqrt{\frac{\omega_{j+1}}{\omega_{j+2}}}v_{j+1}+\sqrt{\frac{\omega_{j}}{\omega_{j+1}}}v_{j-1},\qquad
j\geq0,
\]
where we set $v_{-1}=0$.  Notice this implies $v_0\neq0$ and hence $v_0=1$.  It follows from (\ref{normalrecur}) by the uniqueness of the solution to the recursion relation that $v_j=\varphi_j(t)$ for all $j = 0,1,2,\ldots$, where $\varphi_j$ is the orthonormal polynomial of degree $j$ corresponding to the matrix $\mcj_\omega$ via \eqref{onpdef}. We conclude that $\{\varphi_n(t)\}_{n=0}^{\infty}\in\ell^2(\bbN_0)$ (this could also be deduced from \cite[Proposition 2.16.2]{Rice}).  Define
\[
a_n=P_n(t),\qquad\qquad n\geq0.
\]
Then $\{a_n\}_{n\geq0}$ satisfies the recursion relation
\[
a_{n+1}=ta_n-\frac{\omega_n}{\omega_{n+1}}a_{n-1},
\]
and we compute the corresponding Poincar\'e polynomial to be $z^2-tz+1$ (since (\ref{etacond}) ensures that the assumptions of Poincar\'e's Theorem are met). Poincar\'e's Theorem implies that \[ \lim_{\nri}\frac{a_{n+1}}{a_{n}} \] exists and must equal one of the roots
\[
\frac{t\pm\sqrt{t^2-4}}{2}.
\]
However, $\{\varphi_n(t)\}_{n=0}^{\infty}\in\ell^2(\bbN_0)$, and therefore it follows that
\[
\limsup_{\nri}|a_n|^{1/n}=\limsup_{\nri}|\varphi_n(t)|^{1/n}\leq1.
\]
We conclude that
\[
\lim_{\nri}\frac{a_{n+1}}{a_{n}}=\frac{t-\sqrt{t^2-4}}{2}<1.
\]
In other words, the sequence $\{a_n\}_{n\geq0}$ decays exponentially, and so the function $f(z;t)$ defined in the statement of the theorem is in fact in the space $H^2_{\omega}(\bbR)$. From the recursion relation satisfied by the sequence $\{a_n\}_{n\geq0}$, one easily deduces that $\{P_n(t)\}_{n=0}^{\infty}$ is a critical point of $\Theta$ and
\[
\frac{\langle f(z;t),zf(z;t)\rangle_\omega}{\|zf(z;t)\|_\omega^2}=\frac{t}{2}.
\]

Conversely, suppose $f(z;t)\in H^2_\omega(\bbR)$, where $t>2$.  It is clear that the Maclaurin coefficients of $f(z;t)$ satisfy the recursion (\ref{ncond}) with $\lambda=t$.  Since $f(z;t)$ is holomorphic in the unit disk and the sequence $\omega$ satisfies (\ref{etacond}), we may apply Poincar\'{e}'s Theorem as above to deduce that the sequence $\{|P_n(t)|\}_{n=0}^{\infty}$ decays exponentially as $\nri$.  Proceeding as above, we conclude that the sequence $\{|\varphi_n(t)|\}_{n=0}^{\infty}$ also decays exponentially and hence lies in $\ell^2(\bbN_0)$, which means $(\varphi_0(t),\varphi_1(t),\ldots)^\top$ is an eigenvector for $\mcj_\omega$ with eigenvalue $t$.
\end{proof}

By combining Theorems \ref{normequ} and \ref{geq2}, we obtain our next result.

\begin{corollary}\label{ext2}
If $\|\mcj_\omega\|>2$, then the supremum defining $\mcu_\omega$ in (\ref{extremal}) is attained and the function
\[
f^*(z)=\sum_{n=0}^{\infty}P_n(\|\mcj_\omega\|)z^n
\]
is extremal.
\end{corollary}

\noindent\textit{Remarks.}  i) The proof of Theorem \ref{geq2} shows that the extremal function $f^*$ has a Maclaurin series that converges in a disk of radius strictly larger than $1$, i.e., $f^*$ is analytic across the unit circle.

ii)  Note that Corollary \ref{ext2} resolves the question of existence of extremal function in Theorem \ref{dalpha}, part (A), thus completing the proof of that theorem.

iii)  Applying similar reasoning as in \cite[Proposition 3.1]{McD} and in \cite[Proposition 4.2]{McD}, we see at once that any extremal function for $\mcu_\omega$ must be of the form $Cf^*(\eitheta z)$ for some $\theta\in\bbR$ and $C\in\bbC\setminus\{0\}$.

\smallskip

If $\|\mcj_\omega\|=2$, then determining whether the supremum in (\ref{extremal}) is a maximum is not easy.  Indeed, the supremum in (\ref{extremal}) is a maximum if and only if
\[
\sum_{n=0}^{\infty}\omega_n|P_n(2)|^2<\infty
\]
(see \cite[Proposition 4.2]{McD} where a similar problem is treated).  We have already seen an example where the supremum is not attained. In general it is difficult to resolve the question of existence of an extremal function.

\smallskip

In the next section, we will use Corollary \ref{ext2} to establish the existence of an extremal function for \eqref{extremal} in the $D_{\alpha}$ ($\alpha < 0$) and $\mathcal{A}^2_{\beta}$ ($\beta>-1$) spaces.

%%%%%%%%%%%%%%%%%%%%%%%%%%%%%%%%%%%%%%%%%
\section{Solving the extremal problem in Dirichlet and Bergman-type spaces}\label{exam}

In this section we will compute the norms of the Jacobi matrices
$\mcj_\omega$ (defined as in (\ref{jsdef}), with $\omega$ satisfying
\eqref{etacond}) for Bergman-type $\mathcal{A}^2_{\beta}$
($\beta>-1$) and Dirichlet-type $D_{\alpha}$ ($\alpha < 0$) spaces.
We solve the problem explicitly in $\mathcal{A}^2_{\beta}$
($\beta>-1$), thus yielding a complete description of the extremal
functions and the minimal modulus of the zeros of optimal
approximants and related weighted reproducing kernels. In
Dirichlet-type spaces $D_{\alpha}$ for negative $\alpha$, we improve
the best known bounds on these norms and establish some numerical
estimates.  In both cases, we will be able to write down a
differential equation that the extremal functions must satisfy.
Observe that the space $D_{\alpha}$ (for $\alpha <0$) is the same
space of functions as $\mathcal{A}^2_{-(\alpha+1)}$ but equipped
with a different yet equivalent norm. Accordingly, in the classical
Bergman space $\mathcal{A}^2_{0} = D_{-1}$, where the two norms
coincide, we solve the problem explicitly obtaining a (rational)
extremal function.

We want to emphasize that the solution to the extremal problem
(\ref{extremal}) is highly dependent on the sequence $\omega$, which
determines the geometry of the space, rather than the underlying
space itself (i.e., the topology of the space). This is already
clear from Theorem \ref{dalpha}: given any space $H^2_{\omega}$ with
a nondecreasing $\omega$, we can modify the norm of the space so
that the extremal problem goes from not attaining its supremum
(which is 1) to attaining a maximum larger than 1 by just changing
the value of $\omega_2$ to be $\omega_1/5$. This sensitive
dependence of the extremal problems on the sequence $\omega$ also
extends to the numerical value of the minimal modulus of the zeros,
and accordingly, zeros of weighted reproducing kernels.

Let us begin by examining some functional equations that the extremal functions must satisfy.

%%%%%%%%%%%%%%%%%%%%%%%%%%%%%%%%%%%%%%%%%
\subsection{Functional equations}\label{functionalequations}

Consider now a fixed sequence $\omega$ and the functions
$\{f(z;t)\}$ given in Theorem \ref{geq2}, where $t$ runs over all of
$\sigma(\mcj_\omega)\cap(2,\infty)$, which we will assume is not
empty.  By Theorem \ref{geq2}, we know that the Maclaurin
coefficients $\{a_k(t)\}_{k=0}^{\infty}$ of $f(z;t)$ satisfy the
recursion relation \eqref{ncond} for all $j \geq 2$, where
$\lambda=t=a_1$ and $a_0=1$.

Multiplying the equation \eqref{ncond} by $z^j$ and summing from $j=2$ to $\infty$ yields
\begin{equation}\label{eqsum}
\sum_{j=2}^{\infty} a_j z^j =t z \sum_{j=2}^{\infty} a_{j-1}z^{j-1} - z^2 \sum_{j=2}^{\infty} \frac{\omega_{j-1}}{\omega_j}a_{j-2}z^{j-2}.
\end{equation}

Define an operator $\mathcal{T}_{\omega}$ acting on $g(z)=\sum_{j=0}^{\infty} c_j z^j$ by
\[
\mathcal{T}_{\omega}(g)(z) = \sum_{j=0}^{\infty}\frac{\omega_{j+1}-\omega_{j+2}}{\omega_{j+2}} c_j z^j.
\]
Notice that $\mathcal{T}_{\omega}$ is bounded and injective on $H^2_{\omega}$ whenever $\omega$ is strictly decreasing.

Thus, equation \eqref{eqsum} becomes
\begin{equation}\label{functionaleq}
f(z;t)(z^2 - t z +1) = 1-z^2 \mathcal{T}_{\omega}(f(z;t))(z).
\end{equation}
Let $t_-= \frac{t - \sqrt{t^2-4}}{2}$ and $t_+= \frac{t + \sqrt{t^2-4}}{2}=1/t_-$ be the two roots of the quadratic polynomial on the left-hand side of \eqref{functionaleq}. When $t\in\sigma(\mcj_\omega)\cap(2,\infty)$, we know $f(z;t)$ and $\mathcal{T}_{\omega}(f(z;t))$ are both analytic in the unit disk, so we must have the following:
\begin{equation}\label{functionallambda}
t_+^2=\mathcal{T}_{\omega}(f(z;t))(t_-).
\end{equation}

One of the main objectives of the present section is to use specific
choices of $\omega$ to solve equation \eqref{functionaleq} and find
an explicit formula for $f(z;t)$ for each $t$ in
$\sigma(\mcj_\omega)\cap(2,\infty)$. In fact, for
$\mathcal{A}^2_{\beta}$, the operator $\mathcal{T}_{\omega}$ has a
simple form and the resulting equation is equivalent to a first
order differential equation that we can solve explicitly.  For
$D_{\alpha}$ where $-\alpha \in \bbN$, the operator
$\mathcal{T}_{\omega}$ is a combination of shifts and differential
operators and the equation will provide us with an $n$-th order
differential equation that is solved by the extremal function $f^*$.
We can only solve this differential equation explicitly for
$\alpha=-1,$ but nevertheless, it gives us qualitative information
about the extremal function.

%%%%%%%%%%%%%%%%%%%%%%%%%%%%%%%%%%%%%%%%%
\subsection{Bergman-type spaces.}\label{secBerg}

The operator $\mathcal{T}_{\omega}$ for the space
$\mathcal{A}^2_{\beta}$, $\beta > -1$, sends a function
$g(z)=\sum_{j=0}^{\infty} c_j z^j$ to
\[
\mathcal{T}_{\omega}(g)(z)=(\beta+1) \sum_{j=0}^{\infty} \frac{c_j z^j}{j+2}.
\]
Hence, taking derivatives with respect to $z$ on both sides of \eqref{functionaleq} gives (where $f'(z;t)$ denotes the derivative with respect to the variable $z$)
\begin{equation}
f(z;t)(2z-t) + f'(z;t) (1-t z +z^2) = - (\beta+1)zf(z;t),
\end{equation}
or, equivalently, after some simple computation,
\begin{equation}\label{functBergman}
\frac{f'(z;t)}{f(z;t)} = \frac{(t - (3+\beta)z)}{(1-t z+z^2)}
= \frac{-\frac{3+\beta}{2}+\frac{(1+\beta)t}{2\sqrt{t^2-4}}}{z-\frac{t-\sqrt{t^2-4}}{2}}+\frac{-\frac{3+\beta}{2}-\frac{(1+\beta)t}{2\sqrt{t^2-4}}}{z-\frac{t+\sqrt{t^2-4}}{2}}.
\end{equation}
To find an explicit formula for $f(z;t)$, we first need to determine the values of $t\in\sigma(\mcj_\omega)\cap(2,\infty)$.  This collection of values is known (as we will see below), but we will determine this set using \eqref{functBergman}.  To do so, we first notice that for every $t>2$ (even those values that are not in $\sigma(\mcj_\omega)$), the function $f(z;t)$ defined in Theorem \ref{geq2} solves the differential equation \eqref{functBergman} in some neighborhood of zero because we have the relation \eqref{eqsum}.  Next, recall that we learned from the proof of Theorem \ref{geq2} that for every $t>2$, the function $f(z;t)$ is analytic in the disk centered at $0$ and with radius $t_-$ or $t_+$ (defined as in Section \ref{functionalequations}).  Invoking Theorem \ref{geq2} again shows that $\sigma(\mcj_\omega)\cap(2,\infty)$ is precisely those values of $t>2$ for which $f(z;t)$ is analytic in $\bbD$.

If $t>2$, then a solution of \eqref{functBergman} that is analytic in $\mathbb{D}$ must have a non-negative integer residue at $\frac{t-\sqrt{t^2-4}}{2}$. Therefore, if $t\in\sigma(\mcj_\omega)\cap(2,\infty)$ we have
\[
-\frac{3+\beta}{2}+\frac{(1+\beta)t}{2\sqrt{t^2-4}}=m\in \mathbb{N}\cup\{0\}.
\]
Solving for $t$ yields $t = \sqrt{4+\frac{(1+\beta)^2}{m^2+(3+\beta)m+2+\beta}}$ for some $m\in \mathbb{N}\cup\{0\}$.  For each $m\in\bbN\cup\{0\}$, define
\begin{align}\label{lambdam}
t_m:=\sqrt{4+\frac{(1+\beta)^2}{m^2+(3+\beta)m+2+\beta}}=\frac{2m+\beta+3}{\sqrt{(m+1)(m+\beta+2)}}.
\end{align}
We have thus shown that $\sigma(\mcj_\omega)\cap(2,\infty)\subseteq\{t_m\}_{m=0}^{\infty}$.  To show the reverse inclusion, notice that for each $m\in\bbN\cup\{0\}$, the corresponding differential equation \eqref{functBergman} is
\[
\frac{f'(z;t_m)}{f(z;t_m)} =\frac{m}{z-\frac{\sqrt{m+1}}{\sqrt{m+\beta+2}}}-\frac{m+3+\beta}{z-\frac{\sqrt{m+\beta+2}}{\sqrt{m+1}}}.
\]
Therefore $f(z;t_m)$ is analytic in $\bbD$ and equal to
\[
f(z;t_m)= \left(1-\frac{z\sqrt{m+\beta+2}}{\sqrt{m+1}}\right)^m\left(1-\frac{z\sqrt{m+1}}{\sqrt{m+\beta+2}}\right)^{-m-(3+\beta)}.
\]

We can confirm that the set $\{\pm\, t_m\}_{m=0}^{\infty}$ does indeed coincide with the discrete spectrum of $\mcj_\omega$ by recalling the description of $\sigma(\mcj_{\omega})$ provided in \cite[Theorem 5.5.2]{Ibook}.  To apply the result in \cite[Theorem 5.5.2]{Ibook}, we must use our parameter $\beta$ to specify the parameters $a$, $b$ and $\lambda$ that determine the recursion coefficients in \cite[Equation (5.5.4)]{Ibook}. Indeed, we set\footnote{We caution the reader that in applying the results from \cite{Ibook}, all of the matrix entries and eigenvalues must be rescaled by a factor of $2$.}
\[
a\mapsto-\left(\frac{1+\beta}{2}\right)<0,\qquad\qquad b\mapsto0,\qquad\qquad
\lambda\mapsto\frac{3+\beta}{2}>0.
\]
The portion of \cite[Theorem 5.5.2]{Ibook} corresponding to what is there called Region II and $a<-b$ tells us that
\[
\sigma(\mcj_\omega) = [-2,2] \cup  \left\{\pm\sqrt{4+\frac{(1+\beta)^2}{m^2+(3+\beta)m+2+\beta}}\right\}_{m\in\bbN\cup\{0\}}
\]
in agreement with the preceding calculation.

Of particular interest is the function $f^*(z)$ given in Corollary \ref{ext2}, which is extremal for the problem \eqref{extremal}.  The norm $\|\mcj_\omega\|$ equals the maximum eigenvalue, which is attained when $m = 0$, so
\[
\|\mcj_\omega\|=\sqrt{4+\frac{(1+\beta)^2}{2+\beta}}=\frac{\beta+3}{\sqrt{\beta+2}}.
\]
If we substitute $\|\mcj_\omega\|$ for $t$ in \eqref{functBergman},
it is easy to check that the function $f^*$ must be given by
\[
f^*(z) =\left(1-\frac{z}{\sqrt{\beta+2}}\right)^{-(\beta+3)}.
\]
We summarize our findings in the following theorem.

\begin{theorem}\label{Bergmansolution}
Let $\beta > -1$. The minimal value of the modulus of a zero of an optimal approximant in $\mathcal{A}^2_{\beta}$ is $\frac{2\sqrt{\beta+2}}{\beta+3}$ and it is attained by the approximant of degree 1 to $1/f^*$ where
\[
f^*(z) = \left(1-\frac{z}{\sqrt{\beta+2}}\right)^{-(\beta+3)}.
\]
In particular, for the classical Bergman space ($\beta = 0$), the minimal modulus of a zero of an optimal approximant is $\frac{2}{3}\sqrt{2}$ and the extremal function $f^*(z)$ is given by $\frac{2\sqrt{2}}{(\sqrt{2}-z)^3}$.  Furthermore, the set $\sigma(\mcj_\omega)\cap(2,\infty)$ is given by $\{t_m\}_{m=0}^{\infty}$ (defined as in (\ref{lambdam})) and for each $m\in\bbN\cup\{0\}$, the solutions $f(z;t_m)$ of the differential equation \eqref{functBergman}, normalized so that $f(0;t_m)=1$, are given by
\[
f(z;t_m)= \left(1-\frac{z\sqrt{m+\beta+2}}{\sqrt{m+1}}\right)^m\left(1-\frac{z\sqrt{m+1}}{\sqrt{m+\beta+2}}\right)^{-m-(3+\beta)}.
\]
\end{theorem}

\noindent\textit{Remark.} Note that the precise value $\frac{2}{3}\sqrt{2}$ of the modulus of the minimal zero in the
Bergman space improves the previous bound $\frac{1}{2} \sqrt{2}$ obtained in \cite{BKLSS}.

\medskip

It is noteworthy, as pointed out to us by A. Sola, that when $\beta
\in\bbN \cup \{0 \}$ the function $f^*(z)$ is a multiple of the
derivative of the reproducing kernel evaluated at $\lambda_0$ for
the space $\mathcal{A}^2_{\beta}$, where $\lambda_0:= \frac{t_0 -
\sqrt{t_0^2-4}}{2}$ is the smaller root of the Poincar\'e polynomial
corresponding to $t_0 = \|\mcj_\omega\|$.    In fact, we have the
following.

\begin{corollary}\label{derivative_kernel}
Let $\beta > -1$ be an integer, let $\lambda_0 =\frac{1}{\sqrt{2+\beta}}$, and let $k(z,w)$ be the reproducing kernel in the space $\mca^2_{\beta}$. Then the extremal function $f^*$ for \eqref{extremal} in $\mathcal{A}^2_{\beta}$ is related to the derivative of $k$ by:
$$
f^*(z) = \lambda_0 \partial_z k(z,\lambda_0).
$$
Moreover, the value of the supremum in \eqref{extremal} is given by
\begin{equation*}%\label{derivativeRep}
\sup_{f \in \mathcal{A}^2_{\beta}}  \frac{| \langle f, zf\rangle_{\omega} |}{\| z f \|_{\omega}^2} =\frac{f^{*\prime}(\lambda_0)}{f^*(\lambda_0) +  \lambda_0f^{*\prime}(\lambda_0) } = \frac{ \partial^2_zk(z,\lambda_0)}{\partial_z k(z,\lambda_0) +  \lambda_0 \partial^2_zk(z,\lambda_0)}\,\Big|_{z=\lambda_0}.
\end{equation*}
\end{corollary}

\begin{proof}
Since $\mathcal{A}^2_{\beta} = H^2_{\omega}$ for the weights $\omega_n = {\beta + n + 1 \choose n}^{-1},$ it is well-known that the reproducing kernel for $\mathcal{A}^2_{\beta}$ is given by
\begin{equation}\label{repker}
    k(z,w) =  \left( 1 - z\overline w\right)^{-(2+\beta)}.
\end{equation}
Therefore, by Theorem \ref{Bergmansolution}, the extremal function is equal to
\begin{align}\label{e-fk}
f^*(z) = c\, \partial_z k(z,w)|_{w=\lambda_0} = c\, (\overline{w} /z) \partial_{\overline w} k(z,w)|_{w=\lambda_0},
\end{align}
where the constant $c= \lambda_0$.

In order to show that
\[
\sup_{f \in \mathcal{A}^2_{\beta}}  \frac{| \langle f, zf \rangle_{\omega} |}{\| z f \|_{\omega}^2} = \frac{f^{*\prime}(\lambda_0)}{f^*(\lambda_0) +  \lambda_0 f^{*\prime}(\lambda_0) },
\]
let us first evaluate the denominator $\| z f \|_{\omega}^2$.

\begin{align*}
\langle zf^*, zf^* \rangle_{\omega} &=
\langle z\partial_z k(z,\lambda_0) , \overline{w} \partial_{\overline w}k(z,w) \rangle_{\omega}|_{w=\lambda_0}  \hspace{.5in} \mbox{by \eqref{e-fk}}\\
&=
\overline{w} \partial_{w} \left( w \partial_w k(w, \lambda_0)\right) |_{w=\lambda_0}\\
&=
\overline{w} \left(\partial_w k(w,\lambda_0) + w \partial_w^2 k(w,\lambda_0) \right)|_{w=\lambda_0}\\
&=
\frac{\overline{w}}{\lambda_0} \left( f^*(w) + w f^{*\prime}(w) \right)|_{w=\lambda_0}\\
&= \left( f^*(\lambda_0) + \lambda_0 f^{*\prime}(\lambda_0) \right).
\end{align*}
From the first to the second line of the computation we used the reproducing property of the derivative of a reproducing kernel: $\partial_w g(w) = \langle g, \partial_{\overline w}k(z,w)\rangle_{\omega}.$

Proceeding in a similar manner for the numerator, we obtain
\[
\langle f^*, zf^* \rangle_{\omega}=\overline{w} \partial^2_w k(w,
\lambda_0) |_{w=\lambda_0}=\frac{\overline{w}}{\lambda_0}
f^{*\prime}(w)|_{w=\lambda_0}=f^{*\prime}(\lambda_0).
\]

For the second equality of the proposition we use \eqref{e-fk} again, and the result follows.
\end{proof}

Notice that this connection with reproducing kernels implies that
the extremal functions for $\mathcal{A}^2_{\beta}$ cannot vanish in
the disk, since $\partial_{z} k(z,w)$ with $k(z,w)$ as in
\eqref{repker} doesn't vanish in $\bbD$, and thus, in the proof of
Theorem \ref{Bergmansolution}, the residue at the potential zero of
$f^*$ inside the disk must actually vanish, which if known a priori,
would simplify the proof.  Finally, although we have an explicit
computation showing the relationship between the extremal quantity
and derivatives of reproducing kernels, it would be interesting to
get a direct a priori argument showing that derivatives of
reproducing kernels evaluated at a particular point associated with
the quadratic extremal problem \eqref{extremal} are, in fact,
extremal functions in $\mathcal{A}^2_{\beta}$ for any non-negative
integer $\beta.$

%%%%%%%%%%%%%%%%%%%%%%%%%%%%%%%%%%%%%%%%%
\subsection{Dirichlet-type Spaces.}

Consider the space $D_{\alpha}$,
where $\omega_n=(n+1)^{\alpha}$ and $\alpha <0$.  In this case, the
matrix $\mcj_\omega$ is given by
\begin{align}\label{d-ist}
\mcj_\omega=
\begin{pmatrix}
0 & \left(\frac{2}{3}\right)^{\alpha/2} & 0 & 0 & \cdots \\
\left(\frac{2}{3}\right)^{\alpha/2} & 0 & \left(\frac{3}{4}\right)^{\alpha/2} & 0 & \cdots \\
0 & \left(\frac{3}{4}\right)^{\alpha/2} & 0 & \left(\frac{4}{5}\right)^{\alpha/2} & \cdots \\
0 & 0 & \left(\frac{4}{5}\right)^{\alpha/2} & 0 & \cdots \\
\vdots & \vdots & \vdots & \vdots & \ddots
\end{pmatrix}.
\end{align}
From Theorem \ref{dalpha} we know that $\|\mcj_\omega\|>2$, so $\mcu_\omega>1$, and the supremum in (\ref{extremal}) is attained. In \cite[Theorem B]{BKLSS}, it was shown that the zeros of optimal polynomial approximants in $D_\alpha$ are outside the disk of radius $2^{\alpha/2}$ when $\alpha<0$. The diagonal elements in the matrix $\mcj_\omega$ above are all $0$.  Therefore, (see \cite[Equation (1.3.29)]{Rice}), $ \|\mcj_\omega\| \leq 2 \sup \{ \left(\frac{n}{n+1}\right)^{\alpha/2}: n = 2, 3, \ldots \} = 2(2/3)^{\alpha/2}.$ Thus, invoking Theorem \ref{normequ}, we have the following improvement of the bound given in \cite{BKLSS}.

\begin{theorem}\label{newb}
If $\alpha<0$, then none of the zeros of any optimal polynomial approximants, and hence, the corresponding weighted reproducing kernels, in the space $D_\alpha$ lie inside the open disk of radius $(3/2)^{\alpha/2}$.
\end{theorem}

Figure $1$ shows a Mathematica plot of numerical estimates for one half of the norm of the matrix $\mcj_\omega$ for a variety of values of $\alpha$ in the range $[-12,0]$.  These estimates were obtained by finding the largest eigenvalue of the upper-left $N\times N$ block of $\mcj_\omega$ for large values of $N\approx 200$.  For comparison, we have also plotted the upper bound given by Theorem \ref{newb}, thus showing that the estimate in Theorem \ref{newb} is certainly not optimal.

For $\alpha = -1$, Theorem \ref{newb} tells us that the open disk of radius $(3/2)^{-1/2}\approx 0.816\ldots$ is free of zeros of optimal polynomial approximants in $D_{-1}$.  The exact value $2\sqrt{2}/3\approx 0.943\ldots$ is provided in Theorem \ref{Bergmansolution}.

\begin{figure}[h!]\label{norms}
  \centering
    \includegraphics[width=0.5\textwidth]{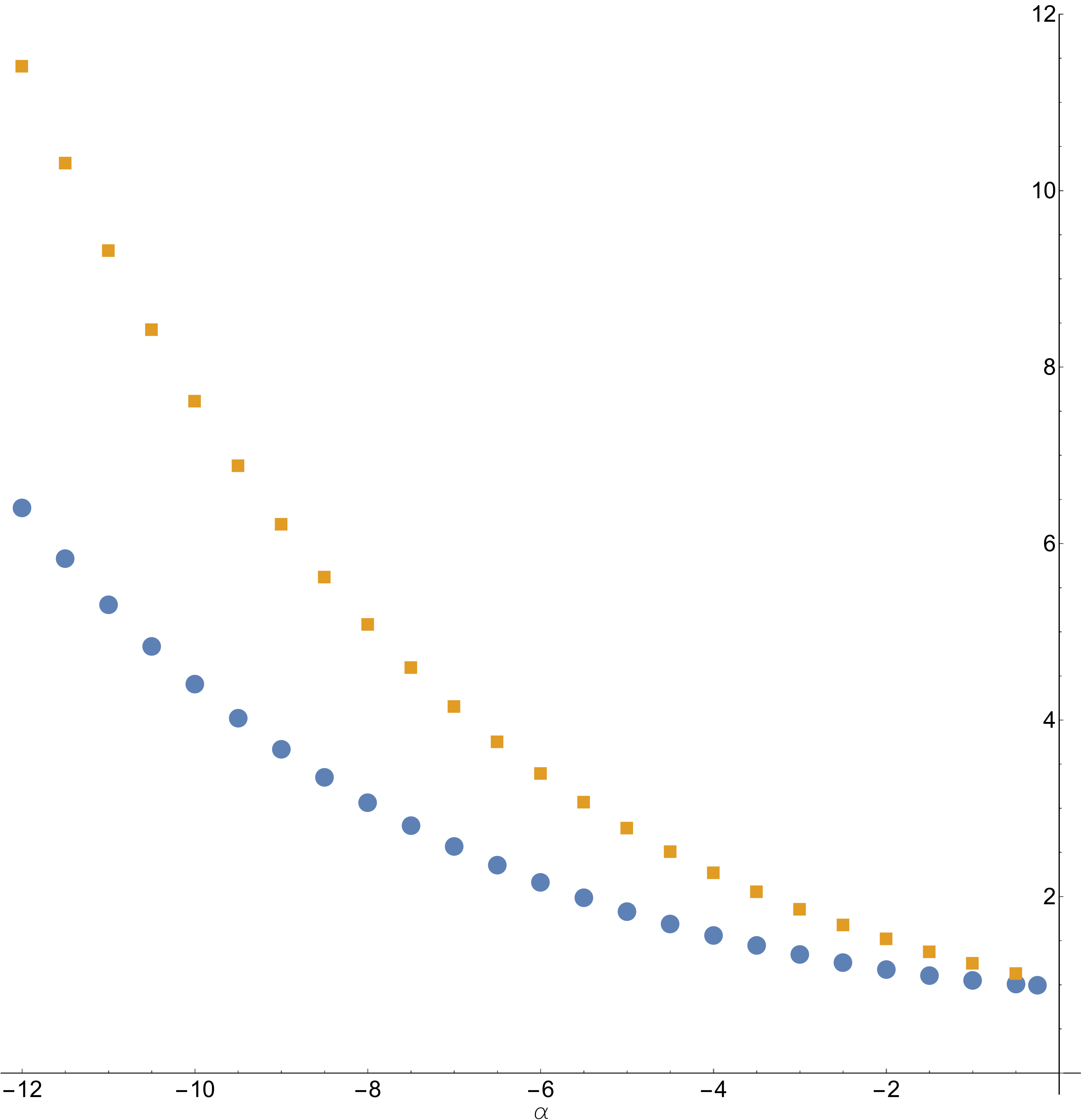}
  \caption{The circles show numerical estimates for one half of the norm of $\mcj_\omega$ for Dirichlet-type spaces $D_\alpha$ for various values of $\alpha\in[-12,0]$.  The squares show the value of the upper bound given by Theorem \ref{newb} for various values of $\alpha\in[-12,0]$.}
\end{figure}

Let us now return to equation \eqref{functionaleq} in the setting of Dirichlet-type spaces.  If $-\alpha\in\bbN$, then we can use the recursion relation satisfied by the Maclaurin coefficients of the extremal function to derive a linear differential equation satisfied by $f^*(z)=f(z,\|\mcj_\omega\|)$ (as in \cite[Section 5.4]{Ibook}).  Doing so yields
\[
(f^{*})'(z) = \|\mcj_\omega\| f^{*}(z) +  \|\mcj_\omega\| z f^{*}(z) -\sum_{j=0}^n {n \choose j} F_j'(z),
\]
where
\begin{align*}
F_j'(z) & = \frac{1}{z} F_{j-1}(z) \qquad \mbox{ for } j \geq 1, \\
F_j(0) & =  0, \\
F_0(z) & =   z^2 f^*(z).
\end{align*}
Performing straightforward algebraic manipulations involving multiplying by $z$ and differentiating, we obtain the following result.

\begin{theorem}\label{fdiff}
If $-\alpha = n\in\bbN$, the function $f^*(z)$ given in Corollary \ref{ext2} satisfies the homogeneous linear differential equation
\[
\sum_{j=0}^{n} r_j(z) (f^*)^{(j)}(z) = 0,
\]
where each $r_j$ is a polynomial, $r_{n-1}(z)= z^{n-2} \left((5+n)z^2- 3 \|\mcj_\omega\| z + 1 \right)$ and $r_n (z) =z^{n-1}(z^2- \|\mcj_\omega\| z + 1)$.
\end{theorem}

We can use Theorem \ref{fdiff} to deduce information about the extremal function $f^*$.  By a classical result for homogeneous linear differential equations (\cite[Chapter XV, p.~356]{Ince}), the singularities of $f^*$ can only occur at the zeros of the leading coefficient $r_n(z)$.  Those zeros are at $0$, $\frac{1}{2}(\|\mcj_\omega\| - \sqrt{ \|\mcj_\omega\|^2-4})$, and $\frac{1}{2}( \|\mcj_\omega\| + \sqrt{ \|\mcj_\omega\|^2-4})$. Since the extremal function $f^*$ is analytic in the unit disk, we can exclude the possibility of a singularity at $0$ or $\frac{1}{2}(\|\mcj_\omega\| - \sqrt{ \|\mcj_\omega\|^2-4})$, so we have arrived at the same conclusion that we saw in the proof of Theorem \ref{geq2}, namely that $f^*$ is analytic in a disk of radius $\frac{1}{2}(\|\mcj_\omega\| + \sqrt{\|\mcj_\omega\|^2-4})$.  In fact, we will see below that in this setting, $f^*$ analytically continues (in general, as a multi-valued function), to $\bbC\setminus\{\frac{1}{2}(\|\mcj_\omega\| + \sqrt{\|\mcj_\omega\|^2-4}) \}$.

Letting $\lambda_0 : = \frac{\|\mcj_\omega\| - \sqrt{\|\mcj_\omega\|^2-4}}{2}$ and $\lambda_1 : = \frac{\|\mcj_\omega\| + \sqrt{\|\mcj_\omega\|^2-4}}{2}$, we can divide by the coefficient $r_n(z)$ and rewrite the differential equation from Theorem \ref{fdiff} as
\[
(f^*)^{(n)}(z) + \frac{p(z,\|\mcj_\omega\|)}{z-\lambda_1} (f^*)^{(n-1)}(z) +  \mbox{lower order terms }  = 0,
\]
where $p(z,\|\mcj_\omega\|):=\frac{ (5+n)z^2- 3 \|\mcj_\omega\| z + 1 }{z ( z - \lambda_0)}$.  By a theorem of Fuchs (\cite{Ince}), $f^*$ can locally be represented as
\[
f^*(z) = \sum_{k=0}^{\infty} b_k (z - \lambda_1)^{k+r},
\]
where $b_0 \neq 0$ and $r$ is a real number.  Since $f^*$ is not entire, we know that $r\notin \mathbb{N}\cup\{0\}$. By taking consecutive derivatives of $f^*$ and equating the lowest order terms in this expansion, we obtain the indicial equation:
\[
r = n - 1 - p(\lambda_1,\|\mcj_\omega\|).
\]
Simplifying this expression results in
\begin{equation}\label{rj}
r = \frac{n}{2} - 3 - \frac{n\|\mcj_\omega\|}{2\sqrt{\|\mcj_\omega\|^2-4}}.
\end{equation}

For $\alpha = -2$ we can use the numerical estimate in Figure 1 to conclude that $r$ is not an integer.  Indeed, using equation \eqref{rj}, one can see that any integer value for $r$ yields a numerical value for $\|\mcj_{\omega}\|$ that we know to be incorrect.  (We expect that the same is true for all negative integers $\alpha$ with $\alpha<-1$.)
 This means that the extremal function $f^*$ in the Dirichlet-type space $D_{-2}$ has a logarithmic singularity of type $(z - \lambda_1)^r,$ where $r$ is a real number that is not an integer. Notice that this implies that Corollary \ref{derivative_kernel} does not apply universally.
 This also shows that identical function spaces but with different norms can yield extremal functions $f^*$ with significant regularity differences.  For the Bergman-type spaces with integer $\beta$, extremal functions are rational functions, while for the Dirichlet-type spaces, when $\alpha$ is a negative integer different from $-1$, we have seen that the extremal functions can have logarithmic singularities.

%%%%%%%%%%%%%%%%%%%%%%%%%%%%%%%%%%%%%%%%%
\section{Jentzsch-Type Theorems}\label{Jentzsch}

Our discussion up to this point has focused on finding the smallest possible modulus of a zero of a polynomial approximant or corresponding weighted reproducing kernel. In this section we will discuss the bulk distribution of zeros of optimal polynomial approximants of large degree and prove theorems analogous to those of Jentzsch.  Recall that the classical theorem of Jentzsch (see, e.g., \cite{DHK,Di,Ti,La}) states that if the Maclaurin series for an analytic function
$f$ has radius of convergence $1$, then every point on the unit circle is a limit point of the zeros of the Taylor polynomials of $f$, and gives a rather precise estimate on the density of those zeros in a neighborhood of the unit circle.

In the context of the spaces we are considering in this paper, it is convenient to introduce the following terminology. We call a positive sequence $\omega=\{\omega_n\}_{n\geq0}$ satisfying (\ref{etacond}) \textit{integral} if there is a regular measure $\mu$ on $\overline{\bbD}$ such that for all $f \in H^2_{\omega}$ we have
\[
\int_{\overline{\bbD}}|f(z)|^2d\mu(z)=\|f\|^2_{\omega}.
\]
In other words, we want the inner product induced by $\omega$ to also be induced by a regular measure on the unit disk so that we may apply the results mentioned in Section \ref{reg}.  We will say that $\omega$ is \textit{quasi-integral} if $\omega$ generates a norm that is equivalent to the norm generated by a regular rotation invariant measure on the closed unit disk.

For any $f\in H^2_{\omega}$, $f \neq 0$, let $W_f$ denote the projection of $1$ onto the space $\overline{f\cdot\mcp}$, where $\mcp$ is the space of polynomials and the closure is taken in the $\|\cdot\|_\omega$-norm.  We note that $f$ is cyclic precisely when $W_f=1$.  Note that since $W_f$ absorbs all the zeros of $f$ in $\bbD$, $W_f/f$ is analytic in $\bbD.$

We can now state our Jentzsch-type results for optimal polynomial approximants (and associated weighted reproducing kernels).

\begin{theorem}\label{J1}
Let $\omega=\{\omega_n\}_{n\geq0}$ be quasi-integral and satisfy (\ref{etacond}). Suppose $f\in H^2_{\omega}$ is such that $W_f/f$ has a singularity on the unit circle and $f(0)\neq0$.  Let $p_n$ denote the $n^{th}$ optimal approximant to $1/f$.  For each $\varepsilon > 0,$ let $\tau_{\varepsilon}(n)$ denote the number of zeros of $p_n$ that lie in the disk $\{ z: |z| \leq 1 + \varepsilon\}$. Then
\[
\limsup_{n \rightarrow \infty} \frac{\tau_{\varepsilon}(n)}{n} = 1.
\]
\end{theorem}

\begin{theorem}\label{J2}
Let $\omega=\{\omega_n\}_{n\geq0}$ be quasi-integral and satisfy (\ref{etacond}). Suppose $f\in H^2_{\omega}$ is such that $W_f/f$ has a singularity on the unit circle and $f(0)\neq0$.  Then every point on the unit circle is a limit point of the zeros of the optimal approximants of $1/f$.
\end{theorem}

\noindent\textit{Remarks.}  (i) Assuming the same hypotheses as in Theorems \ref{J1} and \ref{J2}, the conclusions hold for the weighted reproducing kernels $k_n(z,0)$ when $n \rightarrow \infty$. In many situations, such as in the classical Hardy and Bergman spaces, the reproducing kernels with respect to any logarithmically subharmonic weight do not have any zeros in the open unit disk. This fact, combined with Theorem \ref{J1}, shows that in these spaces, \emph{all} the zeros of the $n$-th optimal approximants (or equivalently of the $n$-th weighted reproducing kernels)  eventually move out of the unit disk.   Moreover, in view of Theorem \ref{J1},
there exists a subsequence $n_k \rightarrow \infty$ for which $n_k-o(n_k)$ zeros of the $n_k$-th polynomial approximant (or, the corresponding $n_k$-th weighted reproducing kernel) converge to the unit circle when $k \rightarrow \infty$.

(ii) Theorems \ref{J1} and \ref{J2} are similar in spirit to Theorem 2.1 of \cite{MSS}, in which the authors considered the behavior of zeros of orthogonal polynomials in weighted Bergman spaces.  We are concerned here with the behavior of zeros of optimal polynomial approximants and associated polynomial reproducing kernels.

(iii) Note that similar to Jentzsch's Theorem, the limit superior in Theorem \ref{J1} cannot be replaced by limit (see \cite{Di}).

(iv) The above theorems apply to functions $f$ that are cyclic but not invertible in $H^2_{\omega}$.

\bigskip

We begin with the following simple observation.

\begin{prop}\label{J3}
Let $K\subseteq\bbD$ be a compact set.  Assuming the same hypotheses stated in Theorem \ref{J1}, let $\tau_K(n)$ be the number of zeros of $p_n$ in $K$.  Then $\tau_K(n)$ is bounded as a function of $n$.
\end{prop}

\begin{proof}
Since $H^2_{\omega}$ is a reproducing kernel Hilbert space, we conclude that $p_nf$ converges uniformly on compact subsets of $\bbD$ to $W_f$.  The conclusion follows from Hurwitz's Theorem (see \cite[Theorem 6.4.1]{Simon2a}).
\end{proof}

Let us now develop some tools we need for the proofs of Theorems \ref{J1} and \ref{J2}.  We begin with some observations about quasi-integral sequences $\omega$.  Let $\mu$ be a regular rotation invariant measure on $\bard$ that generates a norm equivalent to the norm generated by $\omega$.  Let $\varphi_n(z;\mu)$ be the degree
$n$ orthonormal polynomial for the measure $|f|^2\mu$ and let $\varphi_n(z;\omega)$ be the degree $n$ orthonormal polynomial for the space with inner product
\begin{align}\label{fin}
\langle\cdot f,\cdot f\rangle_\omega.
\end{align}
We denote their respective leading coefficients by $\kappa_n(\mu)$ and $\kappa_n(\omega)$ and let $P_n(z;\mu)=\varphi_n(z;\mu)/\kappa_n(\mu)$ and $P_n(z;\omega)=\varphi_n(z;\omega)/\kappa_n(\omega)$.

\begin{lemma}\label{J4}
Assuming the same hypotheses stated in Theorem \ref{J1}, $W_f(0)>0$.
\end{lemma}

\begin{proof}
We have
\[
\lim_{\nri}\|p_nf-W_f\|_\omega=0.
\]
Since we are in a reproducing kernel Hilbert space, this implies
\[
\lim_{\nri}p_n(0)f(0)=W_f(0).
\]
By (\ref{optortho}), this gives
\[
W_f(0)=\lim_{\nri}  |f(0)|^2\sum_{k=0}^{n}|\varphi_k(0;\omega)|^2= |f(0)|^2\sum_{k=0}^{\infty}|\varphi_k(0;\omega)|^2>0
\]
as desired.
\end{proof}

We will need the following lemma, which depends on the rotation invariance of $\mu$.

\begin{lemma}\label{regular2}
The measure $|f|^2\mu$ is regular on $\bard$.
\end{lemma}

\begin{proof}
First note that the extremal property satisfied by the monic orthogonal polynomials (see Section \ref{zero}) implies
\[
\kappa_n(\mu)\geq1,
\]
so we only need to verify the reverse inequality for the $n^{th}$-root asymptotics.  We split the proof into two cases.

\smallskip

\noindent$\cdot$ \underline{Case 1}:  Suppose there exists a non-decreasing sequence $\{r_n\}_{n\in\bbN}$ such that $r_n\rightarrow1$ as $\nri$ and so that $\mu(\{z:|z|=r_n\})>0$ for each $n\in\bbN$.

\smallskip

Let $C_n:=\{z:|z|=r_n\}$.  Then there exists a sequence of positive constants $\{\delta_n\}_{n\in\bbN}$ such that
\[
\mu\geq\mu_n:=\delta_n\chi_{C_n}(z)\frac{d|z|}{2\pi r_n}.
\]
Here we have used that $\mu$ is rotationally invariant. We know that $|f|\neq0$ Lebesgue almost everywhere on $C_n$ (see \cite[Theorem 17.18]{Rudin}) and so it follows from the Erd\"{o}s-Tur\'{a}n criterion that $|f|^2\mu_n$ is a regular measure on $C_n$ for every $n$ (see \cite[Theorem 4.1.1]{StaTo}).  Therefore
\[
\lim_{m \rightarrow \infty} \kappa_m(\mu_n)^{1/m}=\frac{1}{r_n}.
\]
Since $\kappa_m(\mu)\leq\kappa_m(\mu_n)$, we have
\[
\limsup_{m\rightarrow\infty}\kappa_m(\mu)^{1/m}\leq\frac{1}{r_n}.
\]
 Since $r_n$ converges to $1$ as $\nri$, this gives the desired upper bound.

\smallskip

\noindent$\cdot$ \underline{Case 2}:  Suppose there does not exist a
sequence $\{r_n\}_{n\in\bbN}$ as in Case 1.

\smallskip

In this case, there exists a pair of sequences $\{r_n\}_{n\in\bbN}$
and $\{R_n\}_{n\in\bbN}$ in the interval $(0,1)$ so that
\begin{itemize}
\item $r_n<R_n$;
\item $\lim_{\nri}r_n=1$;
\item the intervals $\{[r_n,R_n]\}_{n\in\bbN}$ are pairwise disjoint;
\item $\mu(\{z:r_n\leq|z|\leq R_n\})>0$ for each $n\in\bbN$;
\item there exists a sequence of positive constants $\{\delta_n\}_{n\in\bbN}$ such that $|f(z)|\geq\delta_n$ when $r_n\leq|z|\leq R_n$.
 \end{itemize}
 Then we have
\begin{align*}
\|P_m(z,\mu)\|^2_{L^2(|f|^2\mu)}&=\int_{\bard}|P_m(z;\mu)f(z)|^2d\mu\geq\int_{\{z:r_n\leq|z|\leq R_n\}}|P_m(z;\mu)f(z)|^2d\mu\\
&\geq\delta_n\int_{\{z:r_n\leq|z|\leq R_n\}}|P_m(z;\mu)|^2d\mu\geq\delta_n\int_{\{z:r_n\leq|z|\leq R_n\}}|z|^{2m}d\mu\\
&\geq\delta_nr_n^{2m}\mu(\{z:r_n\leq|z|\leq R_n\}).
\end{align*}
Note that in the penultimate step, we are using the rotation invariance of $\mu$ to conclude that monomials are the monic orthogonal polynomials over the annulus, and therefore they satisfy the minimal $L^2$ norm property there (see Section \ref{zero}).  Now take both sides of this inequality to the $1/(2m)$ power, send $m\rightarrow\infty$, and recall that $r_n$ can be taken arbitrarily close to $1$ to get that $\liminf \left( \|P_m(z,\mu)\|^2_{L^2(|f|^2\mu)} \right)^{1/m} \geq 1$, which is equivalent to what we had to prove.
\end{proof}

\noindent\textit{Remark.}  In Case 1 of the above proof, we allow for the possibility that $\mu$ is supported on the unit circle (i.e., $r_n\equiv1$), which is why we need to appeal to the uniqueness theorem for boundary values of analytic functions (\cite[Theorem 17.18]{Rudin}).

\bigskip

With Lemma \ref{regular2} in hand, the extremal property from
Section \ref{zero} implies
\begin{align*}
\kappa_n(\omega)^{-1}&=\|P_n(\cdot;\omega)f\|_\omega\asymp\|P_n(\cdot;\omega)f\|_{L^2(\mu)}\geq\|P_n(\cdot;\mu)f\|_{L^2(\mu)}=\kappa_n(\mu)^{-1},\\
\kappa_n(\mu)^{-1}&=\|P_n(\cdot;\mu)f\|_{L^2(\mu)}\asymp\|P_n(\cdot;\mu)f\|_{\omega}\geq\|P_n(\cdot;\omega)f\|_{\omega}=\kappa_n(\omega)^{-1}.
\end{align*}
From this and the regularity of $|f|^2\mu$, it follows (see Section
\ref{reg}) that
\begin{align}\label{newkap}
\lim_{\nri}\kappa_n(\omega)^{1/n}=1.
\end{align}

The equivalence of norms also implies that
\[
\lim_{\nri}\|\varphi_n(\omega)\|^{1/n}_{L^2(\mu)}=1,
\]
and so by \cite[page 66]{StaTo}, we have
\begin{align}\label{newout}
\limsup_{\nri}|\varphi_n(z;\omega)|^{1/n}\leq|z|
\end{align}
uniformly on compact subsets of $\bbC\setminus\bard$.  The same
result also tells us that
\begin{align}\label{newzero}
\limsup_{\nri}|\varphi_n(0;\omega)|^{1/n}\leq1.
\end{align}
 We now turn to the proofs of Theorems \ref{J1} and \ref{J2}.  Note that these proofs don't require the full power of the regularity of $\mu$, but only the relations (\ref{newkap}), (\ref{newout}), and (\ref{newzero}).

\begin{proof}[Proof of Theorem \ref{J1}]
Let $\mu$ be the regular rotation invariant measure on $\bard$ that gives rise to a norm equivalent to the one induced by $\omega$ and let $\{\varphi_n(z)\}_{n\geq0}$ be the orthonormal polynomials for the inner product (\ref{fin}).  We recall from Section \ref{Intro} that, denoting $z_{j,n}$ the zeros of $p_n$, the $n$-th optimal approximant of $1/f$,
\begin{align}\label{pn1}
p_n(z)=\overline{f(0)}\sum_{k=0}^{n}\overline{\varphi_k(0)}\varphi_k(z)=:\overline{f(0)}\overline{\varphi_N(0)}\kappa_N\prod_{j=1}^{N}(z-z_{j,n}),
\end{align}
where $N=\mbox{deg}(p_n)$ and $\kappa_N$ is the leading coefficient of $\varphi_N$.  Since $W_f/f$ has a singularity on the unit circle, it cannot happen that $qf=W_f$ for a polynomial $q$, so after passing to a subsequence if necessary, we may assume that $N=n$.

From Lemma \ref{J4} and equations (\ref{pn1}) and (\ref{newkap}), we find that
\[
1=\lim_{\nri}|W_f(0)/f(0)|^{1/n}=\lim_{\nri}|p_n(0)|^{1/n}=\lim_{\nri}|\varphi_n(0)|^{1/n}\left(\prod_{j=1}^n|z_{j,n}|\right)^{1/n}.
\]
We claim that:
\begin{align}\label{jclaim}
\limsup_{\nri}|\varphi_n(0)|^{1/n}=1.
\end{align}
To prove this claim, assume that
\[
\limsup_{\nri}|\varphi_n(0)|^{1/n}=\gamma<1
\]
(by (\ref{newzero}) we do not need to consider the case $\gamma>1$).  The inequality (\ref{newout}) implies that
$\left\{p_n(z)\right\}_{n\in\bbN}$ is a normal family on the open disk centered at zero and radius $1/\gamma$.  Let $p$ be a limit function so that $p_nf$ converges to $pf$ uniformly on compact subsets of $\bbD$ as $\nri$.  We also have $p_nf\rightarrow W_f$ pointwise in $\bbD$ as $\nri$, so $pf=W_f$ and so $p$ provides an analytic continuation of $W_f/f$ to a disk of radius larger than one, which we are assuming is impossible.  This contradiction proves the claim.

From the claim and Proposition \ref{J3}, it follows that
\begin{equation}\label{liminf}
\liminf_{\nri}\left(\prod_{j=1}^n|z_{j,n}|\right)^{1/n}=1.
\end{equation}
To finish the proof, assume that $\limsup_{n \rightarrow \infty} \frac{\tau_{\varepsilon}(n)}{n} = \gamma < 1.$  Then, if we fix $\delta > 0,$ we infer from Proposition \ref{J3} that for large $n$,
$$ \left(\prod_{j=1}^n|z_{j,n}|\right)^{1/n} \geq (1- \delta)^{\frac{\tau_{\varepsilon}(n)}{n}} (1 + \varepsilon)^{\frac{n-\tau_{\varepsilon}(n)}{n}}.$$  Restricting ourselves to a subsequence where the $\liminf$ in \eqref{liminf} is attained,
since $\delta>0$ was arbitrary, letting $n \rightarrow \infty$, we obtain that $1 \geq (1+ \varepsilon)^{1 - \gamma}$, an obvious contradiction.  The proof is now complete.
%Invoking Proposition \ref{J3} again proves the result.
\end{proof}

\bigskip

The proof of Theorem \ref{J1} gives an additional useful result,
which we state as the following corollary.

\begin{corollary}\label{when}
In the notation of the proof of Theorem \ref{J1}, if $\omega$ is
quasi-integral and $\mcn\subseteq\bbN$ is a subsequence so that
\[
\lim_{{\nri}\atop{n\in\mcn}}|\varphi_n(0)|^{1/n}=1,
\]
then for any $\varepsilon>0$
\[
\lim_{{\nri}\atop{n\in\mcn}}\frac{\tau_{\varepsilon}(n)}{n}=1
\]
and
\begin{align}\label{pone}
\lim_{{\nri}\atop{n\in\mcn}}\left(
\prod_{\{v:p_n(v)=0\}}|v|\right)^{\frac{1}{n}}=1.
\end{align}
\end{corollary}

We now turn our attention to the proof of Theorem \ref{J2}.  Heuristically, the proof of this fact is a simple application of \cite[Lemma 4.3b]{MSS}, but there are three issues we must address.  The first is the fact that it is possible that $\varphi_n(0)=0$ for some values of $n$ (we retain the notation from the proof of Theorem \ref{J1}), so we may not have a sequence of polynomials of every degree.  One can see from the proof of \cite[Lemma 4.3b]{MSS} that
this is not an essential assumption. The second issue is that $p_n$ is not monic, but (\ref{newkap}) and (\ref{jclaim}) tell us that we can divide by the leading coefficient and not change the $n^{th}$ root asymptotics of $p_n$, so this issue is easily dealt with also.  The third issue is more serious and requires us to exclude the possibility of zeros of $p_n$ accumulating to infinity.  It is easy to construct examples in the Hardy space $H^2$ where the zeros do
accumulate to infinity so one cannot rule this out in general.  Therefore, to complete the proof of Theorem \ref{J2} we require Lemma \ref{omit} below, which will allow us to disregard the zeros that may accumulate to infinity without meaningfully affecting the necessary estimates.  Our arguments will rely heavily on what we have already learned during the proof of Theorem \ref{J1}.

We begin by restricting our attention to a subsequence $\mcn\subseteq\bbN$ through which the limit superior in (\ref{jclaim}) is attained and also such that $\mbox{deg}(p_n)=n$ for all $n\in\mcn$.  For every $n\in\mcn$, define the polynomial $q_n$ by
\begin{align}\label{qdef2}
q_n(z)=\frac{p_n(z)}{\overline{\varphi_n(0)}\kappa_n\prod_{\{w:p_n(w)=0,\,|w|>2\}}(z-w)}.
\end{align}
Then $q_n(z)$ is monic and Corollary \ref{when} tells us that
\begin{equation}\label{qdeg}
\lim_{{\nri}\atop{n\in\mcn}}\frac{\mbox{deg}(q_n)}{n}=1.
\end{equation}
By thinning the subsequence $\mcn$ if necessary, we may assume that all of the polynomials $\{q_n\}_{n\in\mcn}$ have distinct degrees.  We will need the following lemma.

\begin{lemma}\label{omit}
Under the assumptions of Theorem \ref{J2}, with $q_n$ defined by
(\ref{qdef2}), we have
\[
\lim_{{\nri}\atop{n\in\mcn}}\left|\frac{q_n(z)}{p_n(z)}\right|^{1/n}=1,\qquad\qquad|z|\leq1
\]
and the convergence is uniform.
\end{lemma}

\begin{proof}
By (\ref{newkap}) and the definition of $\mcn$, it suffices to show
that
\[
\lim_{{\nri}\atop{n\in\mcn}}\prod_{\{w:p_n(w)=0,\,|w|>2\}}|z-w|^{1/n}=1
\]
uniformly on $\bard$.  Since each term in the product has modulus at least $1$, it is clear that the limit inferior is at least $1$.  To obtain the upper-bound for the limit superior, note that for $z \in
\overline{\bbD}$,
\[
\prod_{\{w:p_n(w)=0,\,|w|>2\}}|z-w|\leq\prod_{\{w:p_n(w)=0,\,|w|>2\}}(1+|w|)\leq\prod_{\{w:p_n(w)=0,\,|w|>2\}}(2|w|),
\]
so it suffices to show that the $n^{th}$ root of this last product has limit superior at most one as $\nri$, $n \in \mcn$.

By Corollary \ref{when}, if $\varepsilon>0$ is given, then for all $n\in\mcn$ sufficiently large we have
\[
(1+\varepsilon)^n>\prod_{\{w:p_n(w)=0,\,|w|>2\}}|w| \prod_{\{u:p_n(u)=0,\,|u|\leq2\}}|u|.
\]
Take $\nri$ via a subsequence $\mcn_1\subseteq\mcn$ so that
\[
\lim_{{\nri}\atop{n\in\mcn_1}}\prod_{\{w:p_n(w)=0,\,|w|>2\}}|w|^{1/n}=\limsup_{{\nri}\atop{n\in\mcn}}\prod_{\{w:p_n(w)=0,\,|w|>2\}}|w|^{1/n}=:\chi
\]
(the limit (\ref{pone}) assures us that $\chi<\infty$). Substituting the above equality into the preceding estimate shows that if $n\in\mcn_1$ is sufficiently large, then
\[
(1+\varepsilon)^n\geq C(\chi-\varepsilon)^n(1-\varepsilon)^n
\]
for some constant $C>0$ that does not depend on $n$ (we used Proposition \ref{J3} here).  By taking $n^{th}$ roots of both sides and sending $\nri$ through $\mcn_1$, we get
\[
(1+\varepsilon)\geq(\chi-\varepsilon)(1-\varepsilon).
\]
Since $\varepsilon>0$ was chosen arbitrarily, this shows $\chi\leq1$.  Therefore,
\[
\limsup_{{\nri}\atop{n\in\mcn}}\prod_{\{w:p_n(w)=0,\,|w|>2\}}(2|w|)^{1/n}=\limsup_{{\nri}\atop{n\in\mcn}}2^{(n-\tau_1(n))/n}\prod_{\{w:p_n(w)=0,\,|w|>2\}}|w|^{1/n}=\chi\leq1
\]
by our definition of $\mcn$ and Corollary \ref{when}.
\end{proof}

\begin{proof}[Proof of Theorem \ref{J2}]
The proof of this theorem follows Jentzsch's original outline.  We use the same notation as in the proof of Theorem \ref{J1} and the proof of Lemma \ref{omit}.  Let $\delta>0$ be fixed.  The inequalities (\ref{newout}) and (\ref{newzero}) imply that there is a positive constant $C_\delta$ so that for all $n$ and all $\theta\in\bbR$ we have
\[
|\varphi_n(0)\varphi_n(\eitheta)|\leq C_\delta(1+\delta)^n.
\]
Therefore,
\[
|p_n(\eitheta)|\leq
|f(0)|C_\delta\sum_{m=0}^n(1+\delta)^m=|f(0)|C_\delta\frac{(1+\delta)^{n+1}-1}{\delta}.
\]
Taking $n^{th}$ roots, sending $\nri$, and noting that $\delta>0$
was chosen arbitrarily shows
\begin{align*}
\limsup_{\nri}|p_n(\eitheta)|^{1/n}\leq1.
\end{align*}
Lemma \ref{omit} implies that by choosing the subsequence $\mcn$ so that the hypotheses of Corollary \ref{when} are satisfied and defining $\{q_n\}_{n\in\mcn}$ by (\ref{qdef2}), we have
\begin{align}\label{qbound}
\limsup_{{\nri}\atop{n\in\mcn}}|q_n(\eitheta)|^{1/n}\leq1,\qquad\qquad\lim_{{\nri}\atop{n\in\mcn}}|q_n(0)|^{1/n}=1.
\end{align}

One can check that (\ref{qbound}) combined with \eqref{qdeg} and the fact that each $q_n$ is monic is sufficient to apply \cite[Lemma 4.3b]{MSS}, which tells us that the unique weak* limit point of the normalized zero counting measures of the polynomials $\{q_n\}_{n\in\mcn}$ as $\nri$ is the logarithmic equilibrium measure of the unit disk.  Since the
latter is normalized arc-length measure on the unit circle, the desired conclusion follows.
\end{proof}

\bigskip

\noindent\textbf{Example.} Consider the space $\mathcal{A}^2_{\beta}$ previously discussed, with a norm (and inner product) induced by $\omega$ as defined in \eqref{Bergmannorm}. It follows from the results in Chapters 3 and 4, in particular Corollary 4.1.7 and Section 4.2 in \cite{StaTo} that the associated measure is regular. It is easy to see that $f(z)=1-z$ is cyclic. This means $W_f=1$ and $W_f/f$ has a singularity at $1$.  Therefore, we can apply Theorems \ref{J1} and \ref{J2} and conclude that after restricting to an appropriate subsequence (if necessary), most of
the zeros of the optimal polynomial approximants to $1/f$ are near the unit circle and have arguments that accumulate to every point in $[0,2\pi]$.

Now consider the space $D_\alpha$ for $\alpha<0$.  It is easy to see that the spaces $\mathcal{A}^2_{\beta}$ and $D_\alpha$ have equivalent norms when $\beta=-(\alpha +1)$, so we can draw the same conclusions about optimal polynomial approximants to $1/f$ in the space $D_\alpha$ as we did in $\mathcal{A}^2_{\beta}$.

\bigskip

We conclude this section with the following theorem, which provides a very general set of conditions that indicate a sequence $\omega$ is quasi-integral.

\begin{theorem}\label{applyj}
Suppose
\[
\omega_n=h(n)g(n),
\]
where $g$ and $1/g$ are bounded on $[0,\infty)$, $g$ satisfies $g(n+1)/g(n)\rightarrow1$ as $\nri$, and $h$ satisfies
\[
h(r)=\int_0^{\infty}e^{-rt}d\gamma(t),\qquad\qquad r>0
\]
for some probability measure $\gamma$ on $[0,\infty)$ with $0\in\supp(\gamma)$.  Then the sequence $\{\omega_n\}_{n=0}^{\infty}$ is quasi-integral.
\end{theorem}

\begin{proof}
From the definition of quasi-integral, it suffices to show that the sequence $\{h(n)\}_{n=0}^{\infty}$ is integral.  Notice that $h$ is expressed as the Laplace transform of a positive measure on $[0,\infty)$, so it is a completely monotone function.  It follows that the sequence $\{h(n)\}_{n=0}^{\infty}$ is a completely monotone sequence.  The solution to the Hausdorff Moment Problem (see, e.g., \cite[p.~224-228]{HM}) implies that the possible moments of a measure on $[0,1]$ are determined by the complete monotonicity property and hence this is the sequence of moments for some probability measure $\nu$ on $[0,1]$.  Since $0\in\supp(\gamma)$, we know that the sequence $\{h(n)\}_{n=0}^{\infty}$ does not decay exponentially, so $1\in\supp(\nu)$. By a change of variables, we deduce the existence of a probability measure $\tilde{\nu}$ on $[0,1]$ with $1\in\supp(\tilde{\nu})$ so that
\[
h(n)=\int_0^1r^{2n}d\tilde{\nu}(r).
\]
If we set $d\mu(r\eitheta)=d\tilde{\nu}(r)\times \frac{d\theta}{2\pi}$, then the proof of Lemma \ref{regular2} implies that $\mu$ is regular because $1\in\supp(\tilde{\nu})$.  By construction, $\|z^n\|^2_{L^2(\mu)}=h(n)$, so $\{h(n)\}_{n=0}^{\infty}$ is integral as desired.
\end{proof}

%%%%%%%%%%%%%%%%%%%%%%%%%%%%%%%%%%%%%%%%%
\section{A Special Case}\label{particular}

In this section, we explicitly compute the optimal approximants in $D_0$ for
the function $f(z)=(1-z)^{a}$, $\Real a>0$, thus completing the
calculations in the example in \cite{BKLSS}.

\begin{theorem}\label{2conj}
Let $a$ be a complex number with positive real part, i.e., $\Real a>0$. The
optimal degree $n$ polynomial approximant to  $f(z)=(1-z)^{a}$ in
$D_{0}$ is given by
\[
p_n(z)=\sum_{k=0}^n\left(\binom{a+k-1}{k}\frac{B(n+a+1,\overline
a)}{B(n-k+1,\overline a)}\right)z^k,
\]
where $B(x,y)$ is the standard Euler beta function.
\end{theorem}

\begin{proof}
The main observation that allows us to complete the proof is that
the orthonormal polynomials for the measure $|(1-z)^{a}|^2d|z|$ on
the unit circle are known explicitly.  In the case $\Imag a=0$, a
formula for these polynomials is given in \cite[Example 1]{IW}. More
generally, a formula for these polynomials is provided in
\cite{Sri}.  Let us denote these polynomials by
$\{\varphi_k(z)\}_{k\geq0}$, where $\varphi_k(z)$ has degree exactly
$k$ and let $\varphi_k^*(z)$ be the reversed polynomial defined by
\[
\varphi_k^*(z):=z^k\overline{\varphi_k(1/\overline{z})}.
\]
According to the results in \cite[Section 3]{BKLSS} and the
Christoffel-Darboux formula (see \cite[page 124]{OPUC1}), the
polynomial $p_n$ is given by
\[
p_n(z)=\overline{\varphi_n^*(0)}\varphi_n^*(z).
\]
Using the formula in \cite[Theorem 4.2]{Sri} and also \cite[Theorems
3.1 and 4.1]{Sri}, we can write
\[
p_n(z)=\frac{|\Gamma(a+1)|\,\Gamma(2\Real a+1+n)}{\Gamma(2\Real
a+1)^2\,n!}\,\,\overline{_2F_1(-n,a;2\Real
a+1;1)}\,_2F_1(-n,a;2\Real a+1;1-z).
\]
Using \cite[Equation (8.3.7)]{BW}, we can rewrite this as
\[
p_n(z)=\frac{|\Gamma(1+a+n)|^2}{\Gamma(2\Real
a+1+n)\Gamma(n+1)}\,_2F_1(-n,a;-\overline a-n;z).
\]
The coefficient of $z^k$ in this polynomial is
\[
\frac{\Gamma(a+n+1)\Gamma(a+k)\Gamma(\overline a+n-k+1)}{\Gamma(2\Real
a+n+1)\Gamma(k+1)\Gamma(n-k+1)\Gamma(a)},
\]
which yields the desired formula after some simplification.
\end{proof}

\section{Concluding Remarks}\label{ConcRem}
Let us finish with some concluding remarks and interesting directions for
future research.

\subsection{} We have already mentioned (see Corollary \ref{derivative_kernel} and the subsequent remark) that the extremal functions $f^*$ for $\mathcal{A}^2_{\beta}$, $\beta > -1$, are surprisingly equal to the derivatives of the reproducing kernel for $\mathcal{A}^2_{\beta}$ evaluated at a special value $\lambda_0$. The proof of this fact
 comes from a direct computation; however, it would be desirable to obtain a deeper understanding of this phenomenon from a ``high ground" point of view.  Moreover, these reproducing kernels are explicitly known, and as a consequence of Corollary \ref{derivative_kernel}, we see that the extremal functions $f^*$ do not have any zeros in the disk, which as mentioned earlier would lead to a more streamlined proof of Theorem \ref{Bergmansolution}.
Often, as for instance in the classical Bergman space setting, reproducing kernels with respect to logarithmically subharmonic weights do not vanish (see \cite{DKS,DS}), yet not much is known for the derivatives of reproducing kernels (see \cite{MS}). If Theorem \ref{fdiff} could be qualitatively extended to other $D_{\alpha}$ spaces (for $\alpha <0$ not an integer) and the connection between extremal functions and derivatives of reproducing kernels can be exploited, it would provide a tight grasp on the zeros of the reproducing kernels in those spaces, a promising and challenging path.

\subsection{}  In all of our examples, the extremal functions $f^*$, although analytic in a disk larger than the unit disk, do have a finite radius of convergence (see Theorems \ref{Bergmansolution} and \ref{fdiff}) and thus cannot be entire (as opposed to the extremal solutions in \cite{McD}).  On the other hand, all the extremal functions we calculated explicitly are analytically continuable to the whole complex plane except at one point, where they either have a polar singularity as in Theorem \ref{Bergmansolution} or a logarithmic singularity (and thus are multi-valued) as in Theorem \ref{fdiff}.  Does this phenomenon hold for all $D_{\alpha}$ spaces, when $\alpha < 0$ is not an integer, and can we find an explicit solution to the extremal problem \eqref{extremal} in that case?

\subsection{}  For the Dirichlet-type spaces $D_\alpha$ for $\alpha<0$ we do not know the precise value of $\|\mcj_\omega\|$ except in the one case $\alpha=-1$. We only have numerical computations and upper estimates from Theorem \ref{newb} as displayed in Figure 1. While a precise calculation of this norm for all the spaces $\{D_\alpha\}_{\alpha<0}$ would be ideal, it is also clear from Figure 1 that there is room for improvement in estimates of these norms, especially as $\alpha\rightarrow-\infty$.

\subsection{}  We have shown that in certain cases, e.g., the Bergman type spaces, the zeros of optimal approximants and hence the zeros of the associated weighted reproducing kernels can come into the unit disk,
although they will never penetrate a smaller concentric disk.  In \cite{DKS,DS}, it was observed that the reproducing kernel in the Bergman space with a logarithmically subharmonic weight $w>0$ (such as $w = |f|^p$ for $p > 0$ and $f$ analytic in $\bbD$) \emph{never} vanishes in $\bbD$.  Thus, as $n \rightarrow \infty$, the zeros of the optimal approximants $p_n$ to $1/f$, and hence the zeros of the reproducing kernels in the weighted space $\Pol_n (|f|^2 dA)$, are eventually pushed out of any compact subset of $\bbD$ as $n\rightarrow \infty$ (this was also briefly remarked upon in \cite{BKLSS}). Can we assume then that the same behavior of zeros persists for all $D_{\alpha}$ spaces?  A word of caution, however: there are examples of weighted Bergman spaces (see \cite{HZ}) whose
reproducing kernels may have zeros in the disk.

\subsection{}   We have so far  focused on  the positions of single zeros of
optimal approximants.  If we pick a finite set in $\bbC \backslash
\overline{\bbD}$, $\{z_1,...,z_t\}$, then we can construct a
function which belongs to all the $H^2_\omega$ spaces, for which the
optimal polynomial of degree $t$ has exactly those zeros: it is
enough to take the function $f(z)=1/\prod_{i=1}^t (z-z_i)$. Thus,
for the case of nondecreasing $\omega$ the question is then
completely solved in Theorem \ref{dalpha}: a finite set is
achievable as the zero set of an optimal approximant if and only if
the set is contained in $\bbC \backslash \overline{\bbD}$.  For
other sequences $\omega$, in particular decreasing sequences, the
question of which finite sets are achievable as such zero sets is
open. In the case of the $\mathcal{A}^2_\beta$ spaces, since the
extremal functions we found in Theorem \ref{Bergmansolution} don't
have any zeros and the extremals are essentially unique, it is not
possible to find a pair of zeros in which one of them has the
minimal modulus. For general $\omega$ decreasing to $0$, the problem
of determining all the possible configurations of finite sets seems
quite difficult. We can nevertheless find functions whose optimal
approximants have any number of zeros inside the unit disk. Recall
that monomials $\{z^n\}_{n \in \bbN}$ form an orthogonal basis in
$H^2_\omega$. Therefore,
if a function $f=f(z)$ depends only on $z^s$ for some positive integer $s$,
 then the matrix given by
 $\left(M\right)_{j,k} = \left<z^jf,z^kf\right>$, used in \cite{BCLSS} to compute the
 optimal approximants, has zeros in all the positions for which
 $j-k$ is not divisible by $s$.  One can then
use this fact to show that the optimal approximants
to $1/f$ depend only on $z^s$. Consider for example, the function
$f_{k,n}(z) = z^k T_n\left(\frac{1+z}{1-z}\right)$, which appeared
in the proof of Theorem \ref{dalpha}, and a sequence $\omega$
strictly decreasing to $0$. Let $g_r(z)= f_{k,n}(z^r)$ and let $q$
be the optimal approximant of degree $r$ to $1/g_r$. Using the
constructive method to build the approximants described in
\cite{BKLSS}, it is easy to show that if $z_0$ is any root of $q$,
then
$$|z_0|^r = \frac{\|z^r g_r\|_\omega^2}{|\left<g_r,z^r g_r\right>_\omega|}.$$ By choosing $n$
large enough so that $\omega_{kr+1} > 4 \omega_{nr+kr+1}$, the same
computations as
in Theorem \ref{dalpha} will yield that the $r$ roots of $q$ are
inside the unit disk, and moreover,  all lie on a concentric circle. In fact, since
this optimal approximant only depends on $z^r$,
the zeros will also have equidistributed arguments.
The next natural step is to try to describe the zero
sets of approximants of order 2.  Another challenging route is to try to characterize the possible
positions of the zeros of the approximants of order 1 to $1/f$,
depending on the singularities and zeros  of a function $f \in
H^2_\omega$.

\subsection{}  One could also fix the sequence $\omega$ and investigate the behavior of the roots of higher degree approximants. Of course, here we assume that the function $f$ is chosen such that the sequence of optimal approximants does not stagnate, but consists of polynomials with arbitrarily large degrees $n$. Conditions on $f$ such that this happens are well-known, e.g.~$f$ analytically extendable beyond the unit circle and attaining the value zero on the boundary. For fixed $n$ one could, for example, ask to find $f$ that guarantees the degree $n$ approximant (assuming one with exactly that degree exists) to have a root of minimal modulus, or for the sum of moduli of the roots to be minimal.
This is a different problem than the one we have solved in the current paper, since $f$ might have other roots, and we would be forced to restrict the set of functions over which we are solving the extremal problem to those that have the specified roots.

\subsection{}  It is natural to propose a more detailed study of the convergence of zeros of optimal polynomials to the unit circle touched upon in Section \ref{Jentzsch}.  In particular, it would be interesting to study the geometry of the paths along which the zeros ``run" to the unit circle and the connection between these and the particular
singularities of $1/f$ on $\mathbb{T}$. Some examples and numerics were obtained in \cite{BCLSS} and detailed results are known in the case of the Hardy space $H^2$, but the evidence needed to formulate precise conjectures is still thin.  Yet this avenue presents itself as a promising future inquiry.

\end{document}